\documentclass[10pt]{article}
   \usepackage{graphicx}
   \usepackage{epsfig}
   \usepackage{amssymb}
   \usepackage{amsmath}
   \usepackage{amsthm}
   \newtheorem{lemma}{Lemma}[section]
   \newtheorem{theorem}{Theorem}[section]

   {\end{description}}

   {\hfill $\bullet$ \\}

   \newcommand{\be}{\begin{equation}}
   \newcommand{\ee}{\end{equation}}

\begin{document}
    \title{A Robust Three-Level Time Split High-Order Leapfrog/
    Crank-Nicolson Scheme For Two-Dimensional Sobolev and Regularized Long Wave Equations Arising In Fluid Mechanics}
   \author{Eric Ngondiep$^{\text{\,a\,b}}$}
   \date{$^{\text{\,a\,}}$\small{Department of Mathematics and Statistics, College of Science, Imam Mohammad Ibn Saud\\ Islamic University
        (IMSIU), $90950$ Riyadh $11632,$ Saudi Arabia.}\\
     \text{\,}\\
       $^{\text{\,b\,}}$\small{Hydrological Research Centre, Institute for Geological and Mining Research, 4110 Yaounde-Cameroon.}\\
     \text{,}\\
        \textbf{Email addresses:} ericngondiep@gmail.com/engondiep@imamu.edu.sa}
   \maketitle

   \textbf{Abstract.}
   This paper develops a robust three-level time split high-order Leapfrog/Crank-Nicolson technique for solving the two-dimensional unsteady sobolev and regularized long wave equations arising in fluid mechanics. A deep analysis of the stability and error estimates of the proposed approach is considered using the $L^{\infty}(0,T;H^{2})$-norm. Under a suitable time step requirement, the theoretical studies indicate that the constructed numerical scheme is strongly stable (in the sense of $L^{\infty}(0,T;H^{2})$-norm), temporal second-order accurate and convergence of order $O(h^{\frac{8}{3}})$ in space, where $h$ denotes the grid step. This result suggests that the proposed algorithm is less time consuming, fast and more efficient than a broad range of numerical methods widely discussed in the literature for the considered problem. Numerical experiments confirm the theory and demonstrate the efficiency and utility of the three-level time split high-order formulation.\\
    \text{\,}\\

   \\
   {\bf AMS Subject Classification (MSC). 65M12, 65M06}.

  \section{Introduction}\label{sec1}
   The exploration of analytical solutions of the time-dependent partial differential equations (PDEs) plays a vital role in describing the behavior of different physical and biological phenomena arising in the areas of mathematical biology, fluid dynamics, engineering, chemical theory, bio-modeling and fluid mechanics. A large class of unsteady systems of PDEs have been used to model various problems in chemistry, physics, biology and engineering such as: chemical kinematics, fluid mechanics, electricity, nonstationary process in semiconductors in the presence of sources, propagation of wave and shallow water waves, optical fibers, flow of heat, plasma physics, immunology, quantum mechanics, sobolev and regularized long wave models and biology \cite{3lsa,4lsa,5lsa,6lsa,30ks,1ks,2ks,12ks,16ks,16flw} and references therein. The evolutionary two-dimensional sobolev and regularized long wave problems usually arise in the flow of fluids to explaining the motion of wave in media. This model is associated with the Rossy and drift waves in rotating fluids and plasmas, respectively, and it describes a broad range of applications in different branches in engineering and science \cite{1ks}. Developing both exact and efficient numerical solutions for different types of sobolev and regularized long wave equations is an attractive area of research in applied science.\\

   In this paper, we consider the two-dimensional sobolev and regularized long wave equations defined in \cite{1ks} as
     \begin{equation}\label{1}
      u_{t}-\alpha\Delta u_{t}-\gamma\Delta u+(\beta,\beta)\cdot\nabla u=f(x,y,t,u,u_{x},u_{y}),\text{\,\,\,\,on\,\,\,\,}\Omega\times(0,T],
     \end{equation}
     with initial condition
      \begin{equation}\label{2}
      u(x,y,0)=u_{0}(x,y),\text{\,\,\,\,on\,\,\,\,}\Omega\times\partial\Omega,
     \end{equation}
     and boundary condition
      \begin{equation}\label{3}
      u(x,y,t)=g(x,y,t)\text{\,\,\,\,on\,\,\,\,}\partial\Omega\times[0,T],
     \end{equation}
     where $f(x,y,t,u,u_{x},u_{y})=f_{1}(x,y,t,u,u_{x})+f_{2}(x,y,t,u,u_{y})$, and $f_{m}$, for $m=1,2$, are nonlinear functions. For the sake of stability analysis and error estimates, we assume that the functions $f_{m}$ are locally Lipschitz with respect to the unknown $u$. $\Delta$ and $\nabla$ designate the Laplacian and gradient operators, respectively. $u_{z}$ represents $\frac{\partial u}{\partial z}$, for $z=x,y,t$. $\alpha$, $\beta$ and $\gamma$ are nonnegative constant less than one, with $\alpha\neq0$, $u_{0}$ and $g$ denote the initial and boundary conditions, respectively. Equation $(\ref{1})$ is a third-order mixed derivative in both time and space and it is referred to the sobolev equations arising in flow of liquid through the theory of heat conduction, fissured rocks and non-study flow \cite{4ks,3ks,2ks}. When $\alpha=0$, equation $(\ref{1})$ becomes a nonlinear convection-diffusion-reaction model which has been widely studied in the literature \cite{1en,2en,3en,4en,5en,6en,7en,8en}. For $\alpha\neq0$, a large set of numerical methods have been developed in an approximate solution of the initial-boundary value problem $(\ref{1})$-$(\ref{3})$, such as: Galerkin finite element methods, split least-square plan, Runge Kutta method, conservative scheme, a computational approach, Lumped Galerkin procedure, etc.... For more details, we refer the readers to \cite{5ks,6ks,7ks,8ks,10ks,11ks,13ks,15ks,16ks} and references therein. For these methods, either the stability analysis or the error estimates has not been considered. In this work, we develop a three-level time split high-order Leapfrog/Crank-Nicolson approach for solving the partial differential equation $(\ref{1})$ subjects to initial-boundary conditions $(\ref{2})$-$(\ref{3})$. Under an appropriate time step limitation, the proposed formulation is strongly stable (in the sense of $L^{\infty}(0,T;H^{2})$-norm), temporal second-order accurate and convergence in space with order $O(h^{\frac{8}{3}})$ in the $L^{\infty}(0,T;H^{2})$-norm, where $h$ denotes the space step. This result suggests that the developed approach is faster and more efficient than a broad range of numerical schemes \cite{1ks,ks,6ks,7ks,13ks,11ks,16ks} widely studied in the literature for the considered problem $(\ref{1})$-$(\ref{3})$. Furthermore, the new three-level time split technique should be considered as a strong numerical technique for integrating a general system of nonlinear PDEs.\\

     The highlights of the paper is the following items:
     \begin{description}
      \item[i)] development of the three-level time split high-order Leapfrog/Crank-Nicolson scheme for solving the initial-boundary value problem $(\ref{1})$-$(\ref{3})$,
      \item[ii)] stability analysis and error estimates of the proposed numerical approach,
      \item[iii)] some numerical examples that confirm the theoretical studies.
     \end{description}

     The remainder of the paper is organized as follows. In Section $\ref{sec2}$, we construct the three-level time split high-order Leapfrog/Crank-Nicolson technique for solving the model problem $(\ref{1})$-$(\ref{3})$. Section $\ref{sec3}$ provides a deep analysis of the stability and error estimates of the new algorithm whereas some numerical evidences are considered in  Section $\ref{sec4}$. Section $\ref{sec5}$ presents the general conclusions together with our future investigations.

    \section{Development of the three-level time split approach}\label{sec2}
    In this section, we construct a three-level time split high-order Leapfrog/Crank-Nicolson scheme in a numerical solution of the two-dimensional time dependent sobolev and regularized long wave equation $(\ref{1})$ with initial condition $(\ref{2})$ and boundary condition $(\ref{3})$. The splitting method describes in this work is as follows: separate equation $(\ref{1})$ into two distinct equations allows to write
       \begin{equation}\label{4}
        u_{t}-\alpha u_{txx}-\gamma u_{xx}+\beta u_{x}=f_{1}(x,y,t,u,u_{x}),
       \end{equation}
       \begin{equation}\label{5}
        u_{t}-\alpha u_{tyy}-\gamma u_{yy}+\beta u_{y}=f_{2}(x,y,t,u,u_{y}).
        \end{equation}
       Let $M$ and $N$ be positive integers and $L_{l}$, for $l=1,2,3,4$, be four real numbers such that: $L_{1}<L_{2}$ and $L_{3}<L_{4}$. We set $\Omega=(L_{1},L_{2})\times(L_{3},L_{4})$, be the region of fluid, $k=\frac{T}{N}$, $h_{x}=\frac{L_{2}-L_{1}}{M}$ and $h_{y}=\frac{L_{4}-L_{3}}{M}$, be the time step and mesh steps in the $x$-direction and $y$-direction, respectively. For the convenience of writing we set $u^{n}_{ij}=u(x_{i},y_{j},t_{n})$ and $U^{n}_{ij}=U(x_{i},y_{j},t_{n})$, be the analytical solution and the approximate one, respectively, at the discrete point $(x_{i},y_{j},t_{n})$, where $x_{i}=L_{1}+ih_{x}$, $y_{j}=L_{3}+jh_{y}$, and $t_{n}=nk$, for $i,j=0,1,2,...,M$, and $n=0,1,2,...,N$. In addition, suppose $\Omega_{k}=\{t_{n},\text{\,\,\,}n=0,1,2,...,N\}$, $\overline{\Omega}_{h_{xy}}=\{(x_{i},y_{j}),\text{\,\,} i,j=0,1,...,M\}$, $\Omega_{h_{xy}}=\overline{\Omega}_{h_{xy}}\cap\Omega$ and $\partial\Omega_{h_{xy}}=\partial\overline{\Omega}_{h_{xy}}\cap\Omega$, be regular partitions of domains: $[0,T]$, $\overline{\Omega}$, $\Omega$ and $\partial\Omega$, respectively. Thus, the space of mesh functions defined over the domain $\Omega_{h_{xy}}\times\Omega_{k}$ is given by $\mathcal{U}_{h_{xyk}}=\{v_{ij}^{n},\text{\,}0\leq i,j\leq M;\text{\,}0\leq n\leq N\}$.  Furthermore, the value of the source term $f_{m}(x,y,t,u,u_{z})$, for $m=1,2$, $z\in\{x,y\}$, at the grid point $(x_{i},y_{j},t_{n})$ is represented by $f_{m}(x_{i},y_{j},t_{n},u_{ij}^{n},u_{z,ij}^{n})$. Lastly, we consider the following linear operators
        \begin{equation*}
        \delta_{t} u_{ij}^{n}=\frac{u_{ij}^{n+\frac{1}{2}}-u_{ij}^{n}}{k/2},\text{\,}\widehat{\delta}_{t} u_{ij}^{n}=
        \frac{u_{ij}^{n}-u_{ij}^{n-\frac{1}{2}}}{k/2},\text{\,}\delta_{x}^{2}u_{ij}^{n}=\frac{u_{i+1,j}^{n}-2u_{ij}^{n}+u_{i-1,j}^{n}}{h_{x}^{2}}
        ,\text{\,}\delta_{y}^{2}u_{ij}^{n}=\frac{u_{i,j+1}^{n}-2u_{ij}^{n}+u_{i,j-1}^{n}}{h_{y}^{2}},
       \end{equation*}
       \begin{equation*}
        \delta_{x}u_{i-\frac{1}{2},j}^{n}=\frac{u_{ij}^{n}-u_{i-1,j}^{n}}{h_{x}},\text{\,}\delta_{x}u_{i+\frac{1}{2},j}^{n}=
        \frac{u_{i+1,j}^{n}-u_{ij}^{n}}{h_{x}},\text{\,} \delta_{y}u_{i,j-\frac{1}{2}}^{n}=\frac{u_{ij}^{n}
        -u_{i,j-1}^{n}}{h_{y}},\text{\,}\delta_{y}u_{i,j+\frac{1}{2}}^{n}=\frac{u_{i,j+1}^{n}-u_{ij}^{n}}{h_{y}}.
       \end{equation*}
       \begin{equation*}
        \delta_{x}^{4}u_{ij}^{n}=\frac{1}{12h_{x}}\left[-u_{i-2,j}^{n}+8u_{i-1,j}^{n}-8u_{i+1,j}^{n}+u_{i+2,j}^{n}\right],\text{\,}
        \delta_{y}^{4}u_{ij}^{n}=\frac{1}{12h_{y}}\left[-u_{i,j-2}^{n}+8u_{i,j-1}^{n}-8u_{i,j+1}^{n}+u_{i,j+2}^{n}\right],
       \end{equation*}
       \begin{equation*}
        \delta_{2x}^{4}u_{ij}^{n}=\frac{1}{12h_{x}^{2}}\left[-u_{i-2,j}^{n}+16u_{i-1,j}^{n}-30u_{ij}^{n}+16u_{i+1,j}^{n}-u_{i+2,j}^{n}\right],
       \end{equation*}
       \begin{equation}\label{6}
       \delta_{2y}^{4}u_{ij}^{n}=\frac{1}{12h_{y}^{2}}\left[-u_{i,j-2}^{n}+16u_{i,j-1}^{n}-30u_{ij}^{n}+16u_{i,j+1}^{n}-u_{i,j+2}^{n}\right].
       \end{equation}
       Furthermore, we introduce the following discrete norms
      \small{\begin{equation*}
        \|u^{n}\|_{2}=\sqrt{h_{x}h_{y}\underset{i=2}{\overset{M-2}\sum}\underset{j=2}{\overset{M-2}\sum}(u^{n}_{ij})^{2}},\text{\,\,\,}
        \|\delta_{x}u^{n}\|_{2}=\sqrt{h_{x}h_{y}\underset{i=1}{\overset{M-2}\sum}\underset{j=2}{\overset{M-2}\sum}(\delta_{x}u^{n}_{i+\frac{1}{2},j})^{2}},\text{\,\,\,}
        \|\delta_{y}u^{n}\|_{2}=\sqrt{h_{x}h_{y}\underset{i=2}{\overset{M-2}\sum}\underset{j=1}{\overset{M-2}\sum}(\delta_{y}u^{n}_{i,j+\frac{1}{2}})^{2}},
       \end{equation*}}
       \small{\begin{equation*}
        \|\delta_{x}^{2}u^{n}\|_{2}=\sqrt{h_{x}h_{y}\underset{i=1}{\overset{M-1}\sum}\underset{j=2}{\overset{M-2}\sum}(\delta_{x}^{2}u^{n}_{ij})^{2}},\text{\,}
         \|\delta_{y}^{2}u^{n}\|_{2}=\sqrt{h_{x}h_{y}\underset{i=2}{\overset{M-2}\sum}\underset{j=1}{\overset{M-1}\sum}(\delta_{x}^{2}u^{n}_{ij})^{2}},\text{\,}
         \||u^{n}|\|_{2,\infty}=\underset{0\leq n\leq N}{\max}\|u^{n}\|_{2},
       \end{equation*}}
       \small{\begin{equation}\label{dn}
         \|u^{n}\|_{H^{2}}=\sqrt{\|u^{n}\|_{2}^{2}+\alpha\left[\|\delta_{x}u^{n}\|_{2}^{2}+\|\delta_{y}u^{n}\|_{2}^{2}+12^{-1}h^{2}
         \left(\|\delta_{x}^{2}u^{n}\|_{2}^{2}+\|\delta_{y}^{2}u^{n}\|_{2}^{2}\right)\right]},\text{\,\,\,}\||u|\|_{H^{2},\infty}=\underset{0\leq n\leq N}{\max}\|u^{n}\|_{H^{2}}.
       \end{equation}}
       In addition, the Hilbert space $L^{2}(\Omega)$ is equipped with the scalar product $\left(\cdot,\cdot\right)_{2}$ defined as
       \begin{equation*}
        \left(u^{n},v^{n}\right)_{2}=h_{x}h_{y}\underset{i=2}{\overset{M-2}\sum}\underset{j=2}{\overset{M-2}\sum}u_{ij}^{n}v_{ij}^{n},\text{\,\,}
        \left(\delta_{x}u^{n},v^{n}\right)_{2}=h_{x}h_{y}\underset{i=1}{\overset{M-2}\sum}\underset{j=2}{\overset{M-2}\sum}
        \delta_{x}u_{i+\frac{1}{2},j}^{n}v_{ij}^{n},
     \end{equation*}
     \begin{equation*}
       \left(\delta_{y}u^{n},v^{n}\right)_{2}=h_{x}h_{y}\underset{i=2}{\overset{M-2}\sum}\underset{j=1}{\overset{M-2}\sum}
       \delta_{y}u_{i,j+\frac{1}{2}}^{n}v_{ij}^{n},\text{\,\,}\left(\delta_{x}u^{n},\delta_{x}v^{n}\right)_{2}=h_{x}h_{y}\underset{i=1}{\overset{M-2}\sum}
       \underset{j=2}{\overset{M-2}\sum}\delta_{x}u_{i+\frac{1}{2},j}^{n}\delta_{x}v_{i+\frac{1}{2},j}^{n},
     \end{equation*}
       \begin{equation}\label{sp}
              \left(\delta_{y}^{2}u^{n},\delta_{y}^{2}v^{n}\right)_{2}=h_{x}h_{y}\underset{i=2}{\overset{M-2}\sum}\underset{j=1}{\overset{M-1}\sum}
              \delta_{y}^{2}u_{ij}^{n}\delta_{y}^{2}v_{ij}^{n}\text{\,\,\,and\,\,\,}\left(\delta_{x}^{2}u^{n},\delta_{x}^{2}v^{n}\right)_{2}
              =h_{x}h_{y}\underset{i=1}{\overset{M-1}\sum}\underset{j=2}{\overset{M-2}\sum}\delta_{x}^{2}u_{ij}^{n}\delta_{x}^{2}v_{ij}^{n}.
       \end{equation}
       In a similar manner, one defines the terms $\left(u^{n},\delta_{x}v^{n}\right)_{2}$, $\left(u^{n},\delta_{y}v^{n}\right)_{2}$ and $\left(\delta_{y}u^{n},\delta_{y}v^{n}\right)_{2}$. The spaces $L^{2}(\Omega)$, $H^{2}(\Omega)$ and $L^{\infty}(0,T;H^{2})$ are equipped with the norms: $\|\cdot\|_{2}$, $\|\cdot\|_{H^{2}}$ and $\||\cdot|\|_{H^{2},\infty}$, respectively, whereas the Hilbert space $L^{2}(\Omega)$, is endowed with the scalar product $\left(\cdot,\cdot\right)_{2}$.\\

        It is worth mentioning that a three-level time split Leapfrog/Crank-Nicolson technique splits the combined numerical scheme into a sequence of one-dimensional operators, thus providing a less time stability requirement. Specifically, the splitting should advance the solution in each direction with a maximum allowable time step. This suggests that the proposed scheme has to be more efficient than the time-split MacCormack procedure. For more details about the time-split MacCormack method, the readers can consult the works discussed in \cite{9en,10en,12en,13en}.\\

        Applying equation $(\ref{4})$ at the grid point $(x_{i},y_{j},t_{n})$ to obtain
        \begin{equation}\label{7}
        u_{t,ij}^{n}-\alpha u_{txx,ij}^{n}-\gamma u_{xx,ij}^{n}+\beta u_{x,ij}^{n}=f_{1}(x_{i},y_{j},t_{n},u_{ij}^{n},u_{x,ij}^{n}).
       \end{equation}
        Expanding the Taylor series for the function $u$ at the grid point $(x_{i},y_{j},t_{n})$ using forward and backward difference representations to get
        \begin{equation*}
          u^{n+\frac{1}{2}}_{ij}=u_{ij}^{n}+\frac{k}{2}u_{t,ij}^{n}+\frac{k^{2}}{8}u_{2t,ij}^{n}+O(k^{3})\text{\,\,\,\,and\,\,\,\,}
        u^{n-\frac{1}{2}}_{ij}=u_{ij}^{n}-\frac{k}{2}u_{t,ij}^{n}+\frac{k^{2}}{8}u_{2t,ij}^{n}+O(k^{3}).
     \end{equation*}
        Subtracting the second equation from the first one, this gives
      \begin{equation*}
      u^{n+\frac{1}{2}}_{ij}-u^{n-\frac{1}{2}}_{ij}=ku_{t,ij}^{n}+O(k^{3}).
      \end{equation*}
       This equation can be rewritten as
      \begin{equation*}
       u_{t,ij}^{n}=\frac{1}{k}(u^{n+\frac{1}{2}}_{ij}-u^{n-\frac{1}{2}}_{ij})+O(k^{2}).
     \end{equation*}
        Plugging this equation and $(\ref{7})$, it is easy to see that
     \begin{equation}\label{8}
      \frac{1}{k}\left(u^{n+\frac{1}{2}}_{ij}-u^{n-\frac{1}{2}}_{ij}\right)-\frac{\alpha}{k}\left(u^{n+\frac{1}{2}}_{xx,ij}-
      u^{n-\frac{1}{2}}_{xx,ij}\right)-\gamma u^{n}_{xx,ij}+\beta u_{x,ij}^{n}=f_{1}(x_{i},y_{j},t_{n},u_{ij}^{n},u_{x,ij}^{n})+O(k^{2}).
     \end{equation}
       In \cite{11en}, with the use of the Taylor series the author has established that
       \begin{equation}\label{9}
      u_{zz,ij}^{q}=\delta_{2z}^{4}u^{q}_{ij}+O(h_{z}^{4}),\text{\,\,\,\,\,}u_{z,ij}^{q}=\delta_{z}^{4}u^{q}_{ij}+O(h_{z}^{4}),
       \end{equation}
      where $z=x,y$, and $q$ is any nonnegative rational number. For $z=x$, substituting equation $(\ref{9})$ into $(\ref{8})$ provides
      \begin{equation*}
       \frac{1}{k}\left(u^{n+\frac{1}{2}}_{ij}-u^{n-\frac{1}{2}}_{ij}\right)-\frac{\alpha}{k}\left(\delta_{2x}^{4}u^{n+\frac{1}{2}}_{ij}-
      \delta_{2x}^{4}u^{n-\frac{1}{2}}_{ij}\right)-\gamma\delta_{2x}^{4}u^{n}_{ij}+\beta\delta_{x}^{4}u_{ij}^{n}=
      f_{1}\left(x_{i},y_{j},t_{n},u_{ij}^{n},\delta_{x}^{4}u_{ij}^{n}+O(h_{x}^{4})\right)+O(k^{2}+k^{-1}h_{x}^{4})),
     \end{equation*}
     which is equivalent to
      \begin{equation}\label{10}
        \left(\mathcal{I}-\delta_{2x}^{4}\right)u^{n+\frac{1}{2}}_{ij}=\left(\mathcal{I}-\delta_{2x}^{4}\right)u^{n-\frac{1}{2}}_{ij}+
        \gamma k\delta_{2x}^{4}u^{n}_{ij}-\beta k\delta_{x}^{4}u_{ij}^{n}+kf_{1}\left(x_{i},y_{j},t_{n},u_{ij}^{n}
        ,\delta_{x}^{4}u_{ij}^{n}+O(h_{x}^{4})\right)+O(k^{3}+h_{x}^{4})),
     \end{equation}
     where $\mathcal{I}$ denotes the identity operator. In addition, the application of the Taylor expansion for the function $f_{m}$, $(m=1,2)$, about the discrete point $\left(x_{i},y_{j},t_{n},u,\delta_{x}^{4}u_{ij}^{n}+O(h_{x}^{4})\right)$ gives
     \begin{equation}\label{10a}
     f_{m}\left(x_{i},y_{j},t_{n},u_{ij}^{n},\delta_{x}^{4}u_{ij}^{n}+O(h_{x}^{4})\right)=f_{m}\left(x_{i},y_{j},t_{n},u_{ij}^{n},
     \delta_{x}^{4}u_{ij}^{n}\right)+O(h_{x}^{4}).
     \end{equation}
     Combining $(\ref{10a})$ and $(\ref{10})$ yield
     \begin{equation}\label{11a}
       \left(\mathcal{I}-\delta_{2x}^{4}\right)u^{n+\frac{1}{2}}_{ij}=\left(\mathcal{I}-\delta_{2x}^{4}\right)u^{n-\frac{1}{2}}_{ij}+
        k\left(\gamma\delta_{2x}^{4}-\beta\delta_{x}^{4}\right)u_{ij}^{n}+kf_{1}\left(x_{i},y_{j},t_{n},u_{ij}^{n},\delta_{x}^{4}u_{ij}^{n}\right)
        +O(k^{3}+h_{x}^{4}+kh_{x}^{4}).
       \end{equation}
        Omitting the infinitesimal term $O(k^{3}+h_{x}^{4}+kh_{x}^{4})$ and replacing the exact solution $"u"$ with the computed one $"U"$ to get the first-step of the desired numerical approach, that is,
        \begin{equation}\label{11}
       \left(\mathcal{I}-\delta_{2x}^{4}\right)U^{n+\frac{1}{2}}_{ij}=\left(\mathcal{I}-\delta_{2x}^{4}\right)U^{n-\frac{1}{2}}_{ij}+
        k\left(\gamma\delta_{2x}^{4}-\beta\delta_{x}^{4}\right)U_{ij}^{n}+kf_{1}\left(x_{i},y_{j},t_{n},U_{ij}^{n},\delta_{x}^{4}U_{ij}^{n}\right).
       \end{equation}
       We recall that the operators $\delta_{2x}^{4}$ and $\delta_{x}^{4}$ are defined in relation $(\ref{6})$.\\

       Applying equation $(\ref{5})$ at the discrete points $(x_{i},y_{j},t_{n+\frac{1}{2}})$ and $(x_{i},y_{j},t_{n+1})$, we obtain
       \begin{equation}\label{12}
        u_{t,ij}^{n+\frac{1}{2}}-\alpha u_{tyy,ij}^{n+\frac{1}{2}}-\gamma u_{yy,ij}^{n+\frac{1}{2}}+\beta u_{y,ij}^{n+\frac{1}{2}}=f_{2}(x_{i},y_{j},t_{n+\frac{1}{2}},u_{ij}^{n+\frac{1}{2}},u_{y,ij}^{n+\frac{1}{2}}),
       \end{equation}
       \begin{equation}\label{13}
        u_{t,ij}^{n+1}-\alpha u_{tyy,ij}^{n+1}-\gamma u_{yy,ij}^{n+1}+\beta u_{y,ij}^{n+1}=f_{2}(x_{i},y_{j},t_{n+1},u_{ij}^{n+1},u_{y,ij}^{n+1}).
       \end{equation}
        The application of the Taylor series for the function $u$ at the mesh point $(x_{i},y_{j},t_{n+\frac{1}{2}})$ with time step $k/2$ using forward and backward difference schemes results in
        \begin{equation*}
          u^{n+1}_{ij}=u_{ij}^{n+\frac{1}{2}}+\frac{k}{2}u_{t,ij}^{n+\frac{1}{2}}+\frac{k^{2}}{8}u_{2t,ij}^{n+\frac{1}{2}}+O(k^{3})\text{\,\,\,\,and\,\,\,\,}
        u^{n+\frac{1}{2}}_{ij}=u_{ij}^{n+1}-\frac{k}{2}u_{t,ij}^{n+1}+\frac{k^{2}}{8}u_{2t,ij}^{n+1}+O(k^{3}).
        \end{equation*}
        Subtracting the second equation from the first one and rearranging terms, this gives
      \begin{equation}\label{14}
      u^{n+1}_{ij}-u^{n+\frac{1}{2}}_{ij}=\frac{k}{4}\left(u_{t,ij}^{n+1}+u_{t,ij}^{n+\frac{1}{2}}\right)+\frac{k^{2}}{16}\left(u_{2t,ij}^{n+\frac{1}{2}}
      -u_{2t,ij}^{n+1}\right)+O(k^{3}).
      \end{equation}
      Utilizing the Mean Value theorem, it holds $u_{2t,ij}^{n+\frac{1}{2}}-u_{2t,ij}^{n+1}=O(k)$. This fact combined with equation $(\ref{14})$ provide
      \begin{equation*}
      u^{n+1}_{ij}-u^{n+\frac{1}{2}}_{ij}=\frac{k}{4}\left(u_{t,ij}^{n+1}+u_{t,ij}^{n+\frac{1}{2}}\right)+O(k^{3}).
      \end{equation*}
       Solving this equation for $u_{t,ij}^{n+1}+u_{t,ij}^{n+\frac{1}{2}}$, we get
      \begin{equation}\label{15}
       u_{t,ij}^{n+1}+u_{t,ij}^{n+\frac{1}{2}}=\frac{4}{k}\left(u^{n+1}_{ij}-u^{n+\frac{1}{2}}_{ij}\right)+O(k^{2}).
     \end{equation}
        Plugging equations $(\ref{12})$, $(\ref{13})$ and $(\ref{15})$, simple calculations yield
     \begin{equation*}
      \frac{4}{k}\left(u^{n+1}_{ij}-u^{n+\frac{1}{2}}_{ij}\right)-\frac{4\alpha}{k}\left(u^{n+1}_{yy,ij}-u^{n+\frac{1}{2}}_{yy,ij}\right)
      -\gamma\left(u^{n+1}_{yy,ij}+u^{n+\frac{1}{2}}_{yy,ij}\right)+\beta\left(u^{n+1}_{y,ij}+u^{n+\frac{1}{2}}_{y,ij}\right)=
      f_{2}(x_{i},y_{j},t_{n+1},u_{ij}^{n+1},u_{y,ij}^{n+1})+
     \end{equation*}
     \begin{equation*}
      f_{2}(x_{i},y_{j},t_{n+\frac{1}{2}},u_{ij}^{n+\frac{1}{2}},u_{y,ij}^{n+\frac{1}{2}})+O(k^{2}).
     \end{equation*}
       Utilizing equations $(\ref{9})$ and $(\ref{10a})$, this becomes
      \begin{equation*}
      \frac{4}{k}\left(u^{n+1}_{ij}-u^{n+\frac{1}{2}}_{ij}\right)-\frac{4\alpha}{k}\left(\delta_{2y}^{4}u^{n+1}_{ij}-\delta_{2y}^{4}u^{n+\frac{1}{2}}_{ij}
      \right)-\gamma\left(\delta_{2y}^{4}u^{n+1}_{ij}+\delta_{2y}^{4}u^{n+\frac{1}{2}}_{ij}\right)+\beta\left(\delta_{y}^{4}u^{n+1}_{ij}
      +\delta_{y}^{4}u^{n+\frac{1}{2}}_{ij}\right)=
      \end{equation*}
     \begin{equation*}
     f_{2}(x_{i},y_{j},t_{n+1},u_{ij}^{n+1},\delta_{y}^{4}u_{ij}^{n+1})+f_{2}(x_{i},y_{j},t_{n+\frac{1}{2}},u_{ij}^{n+\frac{1}{2}},
     \delta_{y}^{4}u_{ij}^{n+\frac{1}{2}})+O(k^{2}+k^{-1}h_{y}^{4}+h_{y}^{4}).
     \end{equation*}
      This is equivalent to
      \begin{equation*}
       \left(\mathcal{I}-\alpha\delta_{2y}^{4}\right)u^{n+1}_{ij}-\frac{k}{4}\left(\gamma\delta_{2y}^{4}-\beta\delta_{y}^{4}\right)u^{n+1}_{ij}=
       \left(\mathcal{I}-\alpha\delta_{2y}^{4}\right)u^{n+\frac{1}{2}}_{ij}-\frac{k}{4}\left(\gamma\delta_{2y}^{4}-\beta\delta_{y}^{4}\right)
       u^{n+\frac{1}{2}}_{ij}+
     \end{equation*}
     \begin{equation}\label{16}
     \frac{k}{4}f_{2}(x_{i},y_{j},t_{n+1},u_{ij}^{n+1},\delta_{y}^{4}u_{ij}^{n+1})+
      \frac{k}{4}f_{2}(x_{i},y_{j},t_{n+\frac{1}{2}},u_{ij}^{n+\frac{1}{2}},\delta_{y}^{4}u_{ij}^{n+\frac{1}{2}})+O(k^{3}+h_{y}^{4}+kh_{y}^{4}).
     \end{equation}
      Tracking the error term $O(k^{3}+h_{y}^{4}+kh_{y}^{4})$ into equation $(\ref{16})$ and replacing the analytical solution $"u"$ with the approximate one $"U"$ to get the second-step of the developed technique, that is,
        \begin{equation*}
       \left(\mathcal{I}-\alpha\delta_{2y}^{4}\right)U^{n+1}_{ij}-\frac{k}{4}\left(\gamma\delta_{2y}^{4}-\beta\delta_{y}^{4}\right)U^{n+1}_{ij}=
       \left(\mathcal{I}-\alpha\delta_{2y}^{4}\right)U^{n+\frac{1}{2}}_{ij}-\frac{k}{4}\left(\gamma\delta_{2y}^{4}-\beta\delta_{y}^{4}\right)
       U^{n+\frac{1}{2}}_{ij}+
     \end{equation*}
     \begin{equation}\label{17}
     4^{-1}kf_{2}(x_{i},y_{j},t_{n+1},U_{ij}^{n+1},\delta_{y}^{4}U_{ij}^{n+1})+
      4^{-1}kf_{2}(x_{i},y_{j},t_{n+\frac{1}{2}},U_{ij}^{n+\frac{1}{2}},\delta_{y}^{4}U_{ij}^{n+\frac{1}{2}}),
     \end{equation}
     where the operators $\delta_{2y}^{4}$ and $\delta_{y}^{4}$ are defined in relation $(\ref{6})$.

     Now, for $n=0$, expanding the Taylor series and using the Mean Value theorem, it is not hard to show that
    \begin{equation*}
       u_{t,ij}^{\frac{1}{2}}+u_{t,ij}^{0}=\frac{4}{k}\left(u^{\frac{1}{2}}_{ij}-u^{0}_{ij}\right)+O(k^{2}).
     \end{equation*}
      Substituting this equation into $(\ref{7})$ and utilizing equations $(\ref{9})$ and $(\ref{10a})$ result in
      \begin{equation*}
       u^{\frac{1}{2}}_{ij}-u^{0}_{ij}-\alpha\left(\delta_{2x}^{4}u^{\frac{1}{2}}_{ij}-\delta_{2x}^{4}u^{0}_{ij}\right)-\frac{\gamma k}{4}\left(\delta_{2x}^{4}u^{\frac{1}{2}}_{ij}+\delta_{2x}^{4}u^{0}_{ij}\right)+\frac{\beta k}{4}\left(\delta_{x}^{4}u^{\frac{1}{2}}_{ij}+\delta_{x}^{4}u^{0}_{ij}\right)=
      \end{equation*}
     \begin{equation}\label{18}
     \frac{k}{4}f_{1}(x_{i},y_{j},t_{\frac{1}{2}},u_{ij}^{\frac{1}{2}},\delta_{x}^{4}u_{ij}^{\frac{1}{2}})+\frac{k}{4}f_{1}(x_{i},y_{j},t_{0},u_{ij}^{0},
     \delta_{x}^{4}u_{ij}^{0})+O(k^{3}+h_{x}^{4}+kh_{x}^{4}).
     \end{equation}
      Truncating the error term $O(k^{3}+h_{x}^{4}+kh_{x}^{4})$, replacing the exact solution $"u"$ with the numerical one $"U"$ and rearranging terms, this gives,
        \begin{equation*}
       \left(\mathcal{I}-\alpha\delta_{2x}^{4}\right)U^{\frac{1}{2}}_{ij}-\frac{k}{4}\left(\gamma\delta_{2x}^{4}-\beta\delta_{x}^{4}\right)
       U^{\frac{1}{2}}_{ij}=\left(\mathcal{I}-\alpha\delta_{2x}^{4}\right)U^{0}_{ij}-\frac{k}{4}\left(\gamma\delta_{2x}^{4}-\beta\delta_{x}^{4}\right)
       U^{0}_{ij}+
     \end{equation*}
     \begin{equation}\label{19}
     4^{-1}kf_{1}(x_{i},y_{j},t_{\frac{1}{2}},U_{ij}^{\frac{1}{2}},\delta_{x}^{4}U_{ij}^{\frac{1}{2}})+
      4^{-1}kf_{1}(x_{i},y_{j},t_{0},U_{ij}^{0},\delta_{x}^{4}U_{ij}^{0}),
     \end{equation}
     A combination of equations $(\ref{19})$, $(\ref{11})$ and $(\ref{17})$ provides a new three-level time split high-order Leapfrog/Crank-Nicolson technique for solving the two-dimensional sobolev and regularized long wave equation $(\ref{1})$ with initial-boundary conditions $(\ref{2})$-$(\ref{3})$. That is, for $i,j=2,3,...,M-2$,
       \begin{equation*}
       \left(\mathcal{I}-\alpha\delta_{2x}^{4}\right)U^{\frac{1}{2}}_{ij}-\frac{k}{4}\left(\gamma\delta_{2x}^{4}-\beta\delta_{x}^{4}\right)
       U^{\frac{1}{2}}_{ij}=\left(\mathcal{I}-\alpha\delta_{2x}^{4}\right)U^{0}_{ij}-\frac{k}{4}\left(\gamma\delta_{2x}^{4}-\beta\delta_{x}^{4}\right)
       U^{0}_{ij}+
     \end{equation*}
     \begin{equation}\label{s1}
     4^{-1}kf_{2}(x_{i},y_{j},t_{\frac{1}{2}},U_{ij}^{\frac{1}{2}},\delta_{x}^{4}U_{ij}^{\frac{1}{2}})+
      4^{-1}kf_{2}(x_{i},y_{j},t_{0},U_{ij}^{0},\delta_{x}^{4}U_{ij}^{0}),
     \end{equation}
     \begin{equation*}
       \left(\mathcal{I}-\alpha\delta_{2y}^{4}\right)U^{1}_{ij}-\frac{k}{4}\left(\gamma\delta_{2y}^{4}-\beta\delta_{y}^{4}\right)U^{1}_{ij}=
       \left(\mathcal{I}-\alpha\delta_{2y}^{4}\right)U^{\frac{1}{2}}_{ij}-\frac{k}{4}\left(\gamma\delta_{2y}^{4}-\beta\delta_{y}^{4}\right)
       U^{\frac{1}{2}}_{ij}+
     \end{equation*}
     \begin{equation}\label{s2}
     4^{-1}kf_{2}(x_{i},y_{j},t_{1},U_{ij}^{1},\delta_{y}^{4}U_{ij}^{1})+
      4^{-1}kf_{2}(x_{i},y_{j},t_{\frac{1}{2}},U_{ij}^{\frac{1}{2}},\delta_{y}^{4}U_{ij}^{\frac{1}{2}}),
     \end{equation}
     and for $n=1,2,...,N-1$,
     \begin{equation}\label{s3}
       \left(\mathcal{I}-\delta_{2x}^{4}\right)U^{n+\frac{1}{2}}_{ij}=\left(\mathcal{I}-\delta_{2x}^{4}\right)U^{n-\frac{1}{2}}_{ij}+
        k\left(\gamma\delta_{2x}^{4}-\beta\delta_{x}^{4}\right)U_{ij}^{n}+kf_{1}\left(x_{i},y_{j},t_{n},U_{ij}^{n},\delta_{x}^{4}U_{ij}^{n}\right),
       \end{equation}
     \begin{equation*}
       \left(\mathcal{I}-\alpha\delta_{2y}^{4}\right)U^{n+1}_{ij}-\frac{k}{4}\left(\gamma\delta_{2y}^{4}-\beta\delta_{y}^{4}\right)U^{n+1}_{ij}=
       \left(\mathcal{I}-\alpha\delta_{2y}^{4}\right)U^{n+\frac{1}{2}}_{ij}-\frac{k}{4}\left(\gamma\delta_{2y}^{4}-\beta\delta_{y}^{4}\right)
       U^{n+\frac{1}{2}}_{ij}+
     \end{equation*}
     \begin{equation}\label{s4}
     4^{-1}kf_{2}(x_{i},y_{j},t_{n+1},U_{ij}^{n+1},\delta_{y}^{4}U_{ij}^{n+1})+
      4^{-1}kf_{2}(x_{i},y_{j},t_{n+\frac{1}{2}},U_{ij}^{n+\frac{1}{2}},\delta_{y}^{4}U_{ij}^{n+\frac{1}{2}}),
     \end{equation}
      with initial and boundary conditions
     \begin{equation}\label{s5}
      U_{ij}^{0}=u_{0,ij},\text{\,\,\,\,}U_{pj}^{n}=g_{pj}^{n},\text{\,\,\,\,}U_{ip}^{n}=g_{ip}^{n},\text{\,\,\,for\,\,\,}p=0,M,
     \text{\,\,\,and\,\,\,}i,j=0,1,2,...,M,
     \end{equation}
     To begin the algorithm, we should set
     \begin{equation*}
     U_{1j}^{n}=U_{0j}^{n},\text{\,\,}U_{M-1,j}^{n}=U_{M,j}^{n},\text{\,\,\,}
     U_{i1}^{n}=U_{i0}^{n},\text{\,\,}U_{i,M-1}^{n}=U_{i,M}^{n},\text{\,\,for\,\,}0\leq i,j\leq M.
     \end{equation*}

     \section{Stability analysis and error estimates of the proposed three-level approach}\label{sec3}
     This section deals with the analysis of stability and error estimates of the constructed technique $(\ref{s1})$-$(\ref{s5})$ for solving the initial-boundary value problem $(\ref{1})$-$(\ref{3})$, under the time step requirement
    \begin{equation}\label{sR}
    k=\min\left\{h_{x}^{4/3},\text{\,\,}h_{y}^{4/3}\right\}.
      \end{equation}
    For the sake of stability analysis and error estimates, we assume that the mesh spaces $h_{x}$ and $h_{y}$ are equal ($h:=h_{x}=h_{y}$) and the analytical solution $"u"$ satisfies the following regularity condition: $u\in L^{\infty}(0,T;H^{2}(\Omega))$. That is, there is a positive constant $\widehat{C}_{0}$ independent of the time step $k$ and the space step $h$, so that
    \begin{equation}\label{RC}
    \||u|\|_{H^{2},\infty}\leq \widehat{C}_{0},
   \end{equation}
   where $\||\cdot|\|_{H^{2},\infty}$ is the discrete norm defined in relation $(\ref{dn})$.

   \begin{theorem}\label{t} (Stability and error estimates).
     Consider $U$ to be the approximate solution obtained from the new algorithm $(\ref{s1})$-$(\ref{s5})$ and let $u$ be the analytical solution of the two-dimensional evolutional sobolev and regularized long wave equation $(\ref{1})$ with initial and boundary conditions $(\ref{2})$ and $(\ref{3})$ , respectively. Under the time step restriction $(\ref{sR})$, the computed solution $U$ satisfies
    \begin{equation}\label{20}
    \||U|\|_{H^{2},\infty}\leq \widehat{C}_{0}+\widehat{C}(k^{2}+h^{8/3}+h^{4}).
   \end{equation}
     Setting: $e^{n+\frac{1}{2}}=u^{n+\frac{1}{2}}-U^{n+\frac{1}{2}}$ and $e^{n+1}=u^{n+1}-U^{n+1}$, be the errors at the time levels $n+\frac{1}{2}$ and $n+1$, respectively, it holds
    \begin{equation}\label{21}
     \||e|\|_{H^{2},\infty}\leq \widetilde{C}(k^{2}+h^{8/3}).
   \end{equation}
    where $\widehat{C}_{0}>0$ is the constant given in estimate $(\ref{RC})$, $\widehat{C}$ and $\widetilde{C}$ are two positive constant that do not dependent on the grid size $h$ and time step $k$.
   \end{theorem}

    The following Lemmas $\ref{l1}$-$\ref{l3}$ plays an important role in the proof of the main result of this paper (namely, Theorem $\ref{t}$).

      \begin{lemma}\label{l1}
      Suppose $w,v\in\mathcal{C}_{D}^{0,0}$, be two functions defined over the domain $D=\Omega\times[0,T]$ and satisfying: $w(x_{l},y_{j},t)=v(x_{l},y_{j},t)=w(x_{i},y_{l},t)=v(x_{i},y_{l},t)=0$, for $l\in\{0,1,M-1,M\}$, $0\leq i,j\leq M$, and any $t\in[0,T]$. Let $w(x_{i},y_{j},t)=w_{ij}(t)$ and $v(x_{i},y_{j},t)=v_{ij}(t)$, then it holds
     \begin{equation}\label{22}
     \left(\delta_{mz}^{4}w(t),v(t)\right)_{2}\leq C_{3}\|\delta_{z}w(t)\|_{2}\|\delta_{z}v(t)\|_{2},
    \end{equation}
      \begin{equation}\label{22a}
     \left(\delta_{mz}^{4}w(t),v(t)\right)_{2}\leq \frac{1}{2}\left[\|\delta_{z}w(t)\|_{2}^{2}+\|\delta_{z}v(t)\|_{2}^{2}+\frac{h^{2}}{12}
     \left(\|\delta_{z}^{2}w(t)\|_{2}^{2}+\|\delta_{z}^{2}v(t)\|_{2}^{2}\right)\right],
    \end{equation}
     \begin{equation}\label{22aa}
     \left(\delta_{z}^{4}w(t),v(t)\right)_{2}=-\left(\delta_{z}^{4}v(t),w(t)\right)_{2},
    \end{equation}
     for $m=1,2$, and $z=x,y$. $C_{3}$ is a positive constant independent of the mesh size $h$ and the time step $k$ and where we set $\delta_{1z}^{4}=\delta_{z}^{4}$.
   \end{lemma}

       \begin{proof}
       It follows from relation $(\ref{6})$ that
       \begin{equation*}
        \delta_{2x}^{4}w_{ij}(t)=\frac{1}{12h^{2}}\left[-w_{i-2,j}(t)+16w_{i-1,j}(t)-30w_{ij}(t)+16w_{i+1,j}(t)-w_{i+2,j}(t)\right]=
       \end{equation*}
      \begin{equation}\label{23}
       \frac{1}{12h}\left[-(\delta_{x}w_{i-\frac{1}{2},j}(t)-\delta_{x}w_{i-\frac{3}{2},j}(t))+
       14(\delta_{x}w_{i+\frac{1}{2},j}(t)-\delta_{x}w_{i-\frac{1}{2},j}(t))-(\delta_{x}w_{i+\frac{3}{2},j}(t)-\delta_{x}w_{i+\frac{1}{2},j}(t))\right].
       \end{equation}
       Utilizing $(\ref{23})$ and the definition of scalar product defined in $(\ref{sp})$, we obtain
       \begin{equation*}
        \left(\delta_{2x}^{4}w(t),v(t)\right)_{2}=h^{2}\underset{i=2}{\overset{M-2}\sum}\underset{j=2}{\overset{M-2}\sum}
       (\delta_{2x}^{4}w_{ij}(t))v_{ij}(t)=\frac{h}{12}\left\{-\underset{i=2}{\overset{M-2}\sum}\underset{j=2}{\overset{M-2}\sum}
       (\delta_{x}w_{i-\frac{1}{2},j}(t)-\delta_{x}w_{i-\frac{3}{2},j}(t))v_{ij}(t)+\right.
       \end{equation*}
       \begin{equation}\label{24}
        \left.14\underset{i=2}{\overset{M-2}\sum}\underset{j=2}{\overset{M-2}\sum}
       (\delta_{x}w_{i+\frac{1}{2},j}(t)-\delta_{x}w_{i-\frac{1}{2},j}(t))v_{ij}(t)-\underset{i=2}{\overset{M-2}\sum}\underset{j=2}{\overset{M-2}\sum}
       (\delta_{x}w_{i+\frac{3}{2},j}(t)-\delta_{x}w_{i+\frac{1}{2},j}(t))v_{ij}(t)\right\}.
       \end{equation}
       Utilizing the assumption: $w_{lj}(t)=v_{lj}(t)=w_{il}(t)=v_{il}(t)=0$, for any $l=0,1,M-1,M$; $i,j=0,1,2,...,M$, and applying the summation by parts, straightforward computations give
      \begin{equation}\label{25}
        -h\underset{i=2}{\overset{M-2}\sum}\underset{j=2}{\overset{M-2}\sum}(\delta_{x}w_{i-\frac{1}{2},j}(t)-\delta_{x}w_{i-\frac{3}{2},j}(t))
       v_{ij}(t)=h^{2}\underset{i=1}{\overset{M-2}\sum}\underset{j=2}{\overset{M-2}\sum}\delta_{x}w_{i-\frac{1}{2},j}(t)\delta_{x}v_{i+\frac{1}{2},j}(t),
      \end{equation}
      \begin{equation}\label{26}
      14h\underset{i=2}{\overset{M-2}\sum}\underset{j=2}{\overset{M-2}\sum}(\delta_{x}w_{i+\frac{1}{2},j}(t)-\delta_{x}w_{i-\frac{1}{2},j}(t))
     v_{ij}(t)=-14h^{2}\underset{i=1}{\overset{M-2}\sum}\underset{j=2}{\overset{M-2}\sum}\delta_{x}w_{i+\frac{1}{2},j}(t)\delta_{x}v_{i+\frac{1}{2},j}(t),
     \end{equation}
     \begin{equation}\label{27}
     -h\underset{i=2}{\overset{M-2}\sum}\underset{j=2}{\overset{M-2}\sum}(\delta_{x}w_{i+\frac{3}{2},j}(t)-\delta_{x}w_{i+\frac{1}{2},j}(t))
     v_{ij}(t)=\underset{i=1}{\overset{M-2}\sum}\underset{j=2}{\overset{M-2}\sum}\delta_{x}w_{i+\frac{3}{2},j}(t)\delta_{x}v_{i+\frac{1}{2},j}(t).
     \end{equation}
     Substituting equations $(\ref{25})$-$(\ref{27})$ into $(\ref{24})$ to get
     \begin{equation}\label{28}
        \left(\delta_{2x}^{4}w(t),v(t)\right)_{2}=\frac{h^{2}}{12}\left\{\underset{i=1}{\overset{M-2}\sum}\underset{j=2}{\overset{M-2}\sum}
        \delta_{x}w_{i-\frac{1}{2},j}\delta_{x}v_{i+\frac{1}{2},j}-14\underset{i=1}{\overset{M-2}\sum}\underset{j=2}{\overset{M-2}\sum}
        \delta_{x}w_{i+\frac{1}{2},j}\delta_{x}v_{i+\frac{1}{2},j}+\underset{i=1}{\overset{M-2}\sum}\underset{j=2}{\overset{M-2}\sum}
        \delta_{x}w_{i+\frac{3}{2},j}\delta_{x}v_{i+\frac{1}{2},j}\right\}(t).
       \end{equation}
       Using the discrete norm given by $(\ref{dn})$ together with the Cauchy-Schwarz inequality provide
        \begin{equation}\label{29}
        h^{2}\underset{i=1}{\overset{M-2}\sum}\underset{j=2}{\overset{M-2}\sum}\delta_{x}w_{i-\frac{1}{2},j}(t)\delta_{x}v_{i+\frac{1}{2},j}(t)
        \leq \left[\underset{i=1}{\overset{M-2}\sum}\underset{j=2}{\overset{M-2}\sum}
         \left(\delta_{x}w_{i-\frac{1}{2},j}(t)\right)^{2}\right]^{\frac{1}{2}}\left[\underset{i=1}{\overset{M-2}\sum}\underset{j=2}{\overset{M-2}\sum}
         \left(\delta_{x}v_{i+\frac{1}{2},j}(t)\right)^{2}\right]^{\frac{1}{2}}\leq\|\delta_{x}w(t)\|_{2}\|\delta_{x}v(t)\|_{2}.
         \end{equation}
         Analogously
         \begin{equation}\label{30}
        -14h^{2}\underset{i=1}{\overset{M-2}\sum}\underset{j=2}{\overset{M-2}\sum}\delta_{x}w_{i+\frac{1}{2},j}(t)\delta_{x}v_{i+\frac{1}{2},j}(t)
          \leq 14\|\delta_{x}w(t)\|_{2}\|\delta_{x}v(t)\|_{2},
         \end{equation}
         \begin{equation}\label{31}
         h^{2}\underset{i=1}{\overset{M-2}\sum}\underset{j=2}{\overset{M-2}\sum}\delta_{x}w_{i+\frac{3}{2},j}(t)\delta_{x}v_{i+\frac{1}{2},j}(t)
          \leq \|\delta_{x}w(t)\|_{2}\|\delta_{x}v(t)\|_{2}.
         \end{equation}
         Plugging estimates $(\ref{29})$-$(\ref{31})$ and $(\ref{28})$, this results in
         \begin{equation}\label{32}
        \left(\delta_{2x}^{4}w(t),v(t)\right)_{2}\leq \frac{4}{3}\|\delta_{x}w(t)\|_{2}\|\delta_{x}v(t)\|_{2}.
       \end{equation}
       This proves estimate $(\ref{22})$, for $z=x$ and $m=2$. Furthermore, utilizing equation $(\ref{28})$, it is easy to observe that
        \begin{equation}\label{33a}
        \left(\delta_{2x}^{4}w(t),v(t)\right)_{2}=-h^{2}\underset{i=1}{\overset{M-2}\sum}\underset{j=2}{\overset{M-2}\sum}
        \delta_{x}w_{i+\frac{1}{2},j}\delta_{x}v_{i+\frac{1}{2},j}+\frac{h^{2}}{12}\underset{i=1}{\overset{M-2}\sum}\underset{j=2}{\overset{M-2}\sum}
        \left\{\delta_{x}w_{i-\frac{1}{2},j}(t)-2\delta_{x}w_{i+\frac{1}{2},j}(t)+\delta_{x}w_{i+\frac{3}{2},j}(t)\right\}\delta_{x}v_{i+\frac{1}{2},j}(t).
       \end{equation}
       But $\delta_{x}w_{i-\frac{1}{2},j}(t)-2\delta_{x}w_{i+\frac{1}{2},j}(t)+\delta_{x}w_{i+\frac{3}{2},j}(t)=
       h\left(\delta_{x}^{2}w_{i+1,j}(t)-\delta_{x}^{2}w_{ij}(t)\right)$. This fact, together with $(\ref{33a})$ and the scalar product given in $(\ref{sp})$ yield
        \begin{equation*}
        \left(\delta_{2x}^{4}w(t),v(t)\right)_{2}=-\left(\delta_{x}w(t),\delta_{x}v(t)\right)_{2}+\frac{h^{3}}{12}\underset{i=1}{\overset{M-2}\sum}
        \underset{j=2}{\overset{M-2}\sum}\left(\delta_{x}^{2}w_{i+1,j}(t)-\delta_{x}^{2}w_{ij}(t)\right)\delta_{x}v_{i+\frac{1}{2},j}(t).
       \end{equation*}
       Applying the summation by part and utilizing the scalar product $\left(\cdot,\cdot\right)_{2}$ along with the assumption $w_{lj}(t)=w_{il}(t)=0=v_{lj}(t)=v_{il}(t)$, for $l\in\{0,1,M-1,M\}$, $0\leq i,j\leq M$, and for every $t\in[0,T]$, to get
       \begin{equation}\label{33aa}
        \left(\delta_{2x}^{4}w(t),v(t)\right)_{2}=-\left(\delta_{x}w(t),\delta_{x}v(t)\right)_{2}-\frac{h^{2}}{12}
        \left(\delta_{x}^{2}w(t),\delta_{x}^{2}v(t)\right)_{2}.
       \end{equation}
       The application of Cauchy-Schwarz inequality together with estimate $2ab\leq a^{2}+b^{2}$, for every real numbers $a$ and $b$, gives estimate $(\ref{22a})$ for $z=x$.\\

       In a similar way, one easily establishes that
       \begin{equation}\label{33}
        \left(\delta_{2y}^{4}w(t),v(t)\right)_{2}\leq \frac{4}{3}\|\delta_{y}w(t)\|_{2}\|\delta_{y}v(t)\|_{2},
       \end{equation}
       and
       \begin{equation}\label{33aaa}
     \left(\delta_{2y}^{4}w(t),v(t)\right)_{2}\leq \frac{1}{2}\left[\|\delta_{y}w(t)\|_{2}^{2}+\|\delta_{y}v(t)\|_{2}^{2}+\frac{h^{2}}{12}
     \left(\|\delta_{y}^{2}w(t)\|_{2}^{2}+\|\delta_{y}^{2}v(t)\|_{2}^{2}\right)\right].
    \end{equation}
     Utilizing the definition of the linear operator $\delta_{x}^{4}$ defined in relation $(\ref{6})$, it is not hard to see that
     \begin{equation*}
        \delta_{x}^{4}w_{ij}(t)=\frac{1}{12h}\left[-w_{i-2,j}(t)+8w_{i-1,j}(t)-8w_{i+1,j}(t)+w_{i+2,j}(t)\right]=
       \end{equation*}
      \begin{equation*}
       \frac{1}{12}\left[\delta_{x}w_{i-\frac{3}{2},j}(t)-7(\delta_{x}w_{i+\frac{1}{2},j}(t)+\delta_{x}w_{i-\frac{1}{2},j}(t))+
       \delta_{x}w_{i+\frac{3}{2},j}(t)\right].
       \end{equation*}
       So
       \begin{equation}\label{34}
        \left(\delta_{x}^{4}w(t),v(t)\right)_{2}=\frac{h^{2}}{12}\underset{i=2}{\overset{M-2}\sum}\underset{j=2}{\overset{M-2}\sum}\left[
       \delta_{x}w_{i-\frac{3}{2},j}(t)-7(\delta_{x}w_{i+\frac{1}{2},j}(t)+\delta_{x}w_{i-\frac{1}{2},j}(t))+
       \delta_{x}w_{i+\frac{3}{2},j}(t)\right]v_{ij}(t).
       \end{equation}
       Applying the Cauchy-Schwarz and Poincar\'{e}-Friedrich inequalities and using the the hypothesis: $w_{lj}(t)=v_{lj}(t)=w_{il}(t)=v_{il}(t)=0$, for any $l=0,1,M-1,M$; $i,j=0,1,2,...,M$, it is not difficult to observe that
        \begin{equation*}
       h^{2}\underset{i=2}{\overset{M-2}\sum}\underset{j=2}{\overset{M-2}\sum}(\delta_{x}w_{i-\frac{3}{2},j}(t))v_{ij}(t)\leq
       h^{2}\left(\underset{i=2}{\overset{M-2}\sum}\underset{j=2}{\overset{M-2}\sum}(\delta_{x}w_{i-\frac{3}{2},j}(t))^{2}\right)^{\frac{1}{2}}
       \left(\underset{i=2}{\overset{M-2}\sum}\underset{j=2}{\overset{M-2}\sum}(v_{ij}(t))^{2}\right)^{\frac{1}{2}}\leq
       \end{equation*}
       \begin{equation}\label{40}
       \|\delta_{x}w(t)\|_{2}\|v(t)\|_{2}\leq \sqrt{C_{p}}\|\delta_{x}w(t)\|_{2}\|\delta_{x}v(t)\|_{2}.
       \end{equation}
       Analogously,
       \begin{equation}\label{41}
       h^{2}\underset{i=2}{\overset{M-2}\sum}\underset{j=2}{\overset{M-2}\sum}(\delta_{x}w_{i-\frac{1}{2},j}(t))v_{ij}(t)\leq
       \|\delta_{x}w(t)\|_{2}\|v(t)\|_{2}\leq \sqrt{C_{p}}\|\delta_{x}w(t)\|_{2}\|\delta_{x}v(t)\|_{2},
       \end{equation}
       \begin{equation}\label{42}
       h^{2}\underset{i=2}{\overset{M-2}\sum}\underset{j=2}{\overset{M-2}\sum}(\delta_{x}w_{i+\frac{1}{2},j}(t))v_{ij}(t)\leq
       \|\delta_{x}w(t)\|_{2}\|v(t)\|_{2}\leq \sqrt{C_{p}}\|\delta_{x}w(t)\|_{2}\|\delta_{x}v(t)\|_{2},
       \end{equation}
       \begin{equation}\label{43}
       h^{2}\underset{i=2}{\overset{M-2}\sum}\underset{j=2}{\overset{M-2}\sum}(\delta_{x}w_{i+\frac{3}{2},j}(t))v_{ij}(t)\leq
       \|\delta_{x}w(t)\|_{2}\|v(t)\|_{2}\leq \sqrt{C_{p}}\|\delta_{x}w(t)\|_{2}\|\delta_{x}v(t)\|_{2},
       \end{equation}
       where $C_{p}>0$, is a constant independent of the mesh grid $h$ and time step $k$. In the remainder of this paper, we denote by $C_{p}$ be the Poincar\'{e}-Friedrich constant. Substituting estimates $(\ref{40})$-$(\ref{43})$ into equation $(\ref{34})$ provides
       \begin{equation*}
        \left(\delta_{x}^{4}w(t),v(t)\right)_{2}\leq\frac{4\sqrt{C_{p}}}{3}\|\delta_{x}w(t)\|_{2}\|\delta_{x}v(t)\|_{2}.
       \end{equation*}
       Similarly, it is not hard to show that
       \begin{equation*}
        \left(\delta_{y}^{4}w(t),v(t)\right)_{2}\leq\frac{4\sqrt{C_{p}}}{3}\|\delta_{y}w(t)\|_{2}\|\delta_{y}v(t)\|_{2}.
       \end{equation*}
        Furthermore, the application of the summation by parts along with the hypothesis: $w_{lj}(t)=v_{lj}(t)=w_{il}(t)=v_{il}(t)=0$, for $l=0,1,M-1,M$; $i,j=0,1,2,...,M$, and any $t\in[0,T]$, yield
         \begin{equation}\label{35}
     \underset{i=2}{\overset{M-2}\sum}\underset{j=2}{\overset{M-2}\sum}(\delta_{x}w_{i-\frac{3}{2},j}^{q})v_{ij}^{p}=
     -\underset{i=2}{\overset{M-2}\sum}\underset{j=2}{\overset{M-2}\sum}w_{ij}^{q}\delta_{x}v_{i+\frac{3}{2},j}^{p},\text{\,}
     \underset{i=2}{\overset{M-2}\sum}\underset{j=2}{\overset{M-2}\sum}(\delta_{x}w_{i+\frac{3}{2},j}^{q})v_{ij}^{p}=
     -\underset{i=2}{\overset{M-2}\sum}\underset{j=2}{\overset{M-2}\sum}w_{ij}^{q}\delta_{x}v_{i-\frac{3}{2},j}^{p},
     \end{equation}
     \begin{equation}\label{36}
     -\underset{i=2}{\overset{M-2}\sum}\underset{j=2}{\overset{M-2}\sum}(\delta_{x}w_{i-\frac{1}{2},j}(t))v_{ij}(t)=
     \underset{i=2}{\overset{M-2}\sum}\underset{j=2}{\overset{M-2}\sum}w_{ij}(t)\delta_{x}v_{i+\frac{1}{2},j}(t),\text{\,}
     -\underset{i=2}{\overset{M-2}\sum}\underset{j=2}{\overset{M-2}\sum}(\delta_{x}w_{i+\frac{1}{2},j}(t))v_{ij}(t)=
     \underset{i=2}{\overset{M-2}\sum}\underset{j=2}{\overset{M-2}\sum}w_{ij}(t)\delta_{x}v_{i-\frac{1}{2},j}(t),
     \end{equation}
     \begin{equation}\label{37}
     \underset{i=2}{\overset{M-2}\sum}\underset{j=2}{\overset{M-2}\sum}(\delta_{y}w_{i,j-\frac{3}{2}}(t))v_{ij}(t)=
     -\underset{i=2}{\overset{M-2}\sum}\underset{j=2}{\overset{M-3}\sum}w_{ij}(t)\delta_{y}v_{i,j+\frac{3}{2}}(t),\text{\,}
     \underset{i=2}{\overset{M-2}\sum}\underset{j=2}{\overset{M-2}\sum}(\delta_{y}w_{i,j+\frac{3}{2}}(t))v_{ij}(t)=
     -\underset{i=2}{\overset{M-2}\sum}\underset{j=2}{\overset{M-2}\sum}w_{ij}(t)\delta_{y}v_{i,j-\frac{3}{2}}(t),
     \end{equation}
     \begin{equation}\label{38}
     -\underset{i=2}{\overset{M-2}\sum}\underset{j=2}{\overset{M-2}\sum}(\delta_{y}w_{i,j-\frac{1}{2}}(t))v_{ij}(t)=
     \underset{i=2}{\overset{M-2}\sum}\underset{j=2}{\overset{M-2}\sum}w_{ij}(t)\delta_{y}v_{i,j+\frac{1}{2}}(t),\text{\,}
     -\underset{i=2}{\overset{M-2}\sum}\underset{j=2}{\overset{M-2}\sum}(\delta_{y}w_{i,j+\frac{1}{2}}(t))v_{ij}(t)=
     \underset{i=2}{\overset{M-2}\sum}\underset{j=2}{\overset{M-2}\sum}w_{ij}(t)\delta_{y}v_{i,j-\frac{1}{2}}(t).
     \end{equation}
     A combination of equations $(\ref{34})$ and $(\ref{35})$-$(\ref{36})$ gives
    \begin{equation}\label{39}
        \left(\delta_{x}^{4}w(t),v(t)\right)_{2}=\frac{h^{2}}{12}\underset{i=2}{\overset{M-2}\sum}\underset{j=2}{\overset{M-2}\sum}\left[
       -\delta_{x}v_{i+\frac{3}{2},j}(t)+7(\delta_{x}v_{i+\frac{1}{2},j}(t)+\delta_{x}v_{i-\frac{1}{2},j}(t))-
       \delta_{x}v_{i-\frac{3}{2},j}(t)\right]w_{ij}(t)=-\left(\delta_{x}^{4}v(t),w(t)\right)_{2}.
       \end{equation}
       In a similar manner, utilizing the linear operator $\delta_{y}^{4}$ defined in $(\ref{6})$ and equations $(\ref{37})$-$(\ref{38})$, one shows that
       \begin{equation*}
        \left(\delta_{y}^{4}w(t),v(t)\right)_{2}=\frac{h^{2}}{12}\underset{i=2}{\overset{M-2}\sum}\underset{j=2}{\overset{M-2}\sum}\left[
       -\delta_{y}v_{i+\frac{3}{2},j}(t)+7(\delta_{y}v_{i+\frac{1}{2},j}(t)+\delta_{y}v_{i-\frac{1}{2},j}(t))-
       \delta_{y}v_{i-\frac{3}{2},j}(t)\right]w_{ij}(t)=-\left(\delta_{y}^{4}v(t),w(t)\right)_{2}.
       \end{equation*}
    This ends the proof of Lemma $\ref{l1}$.
    \end{proof}

    \begin{lemma}\label{l2}
     Suppose $w\in\mathcal{C}_{D}^{0,0}$, be a function defined on $D=\Omega\times[0,T]$, that satisfies $w_{lj}(t)=w_{il}(t)=0$, for every
       $l=0,1,M-1,M$; $i,j=0,1,...,M$ and $t\in[0,T]$. Letting $w(x_{i},y_{j},t)=w_{ij}(t)$, the following estimates are satisfied:
     \begin{equation}\label{44}
        \left(\delta_{z}^{4}w(t),w(t)\right)_{2}=0,
       \end{equation}
       \begin{equation}\label{45}
        \left(-\delta_{2z}^{4}w(t),w(t)\right)_{2}=\|\delta_{z}w(t)\|_{2}^{2}+\frac{h^{2}}{12}\|\delta_{z}^{2}w(t)\|_{2}^{2},
       \end{equation}
     for every $t\in[0,T]$, where $z=x,y$.
     \end{lemma}

     \begin{proof}
     In equation $(\ref{22aa})$ of Lemma $\ref{l1}$, replacing $v(t)$ with $w(t)$ to get
     \begin{equation*}
    \left(\delta_{z}^{4}w(t),w(t)\right)_{2}=-\left(\delta_{z}^{4}w(t),w(t)\right)_{2},
     \end{equation*}
      which is equivalent to
     \begin{equation*}
     2\left(\delta_{z}^{4}w(t),w(t)\right)_{2}=0.
     \end{equation*}
     Thus, $\left(\delta_{z}^{4}w(t),w(t)\right)_{2}=0$. This proves equation $(\ref{44})$. Now, replacing $v(t)$ with $w(t)$ in equation $(\ref{33aa})$, this provides
      \begin{equation*}
        \left(\delta_{2x}^{4}w(t),w(t)\right)_{2}=-\left(\delta_{x}w(t),\delta_{x}w(t)\right)_{2}-\frac{h^{2}}{12}
        \left(\delta_{x}^{2}w(t),\delta_{x}^{2}w(t)\right)_{2}.
       \end{equation*}
       This is equivalent to
       \begin{equation*}
        \left(-\delta_{2x}^{4}w(t),w(t)\right)_{2}=\|\delta_{x}w(t)\|_{2}^{2}+\frac{h^{2}}{12}\|\delta_{x}^{2}w(t)\|_{2}^{2}.
       \end{equation*}
       Similarly, it is not difficult to show that
       \begin{equation*}
        \left(-\delta_{2y}^{4}w(t),w(t)\right)_{2}=\|\delta_{y}w(t)\|_{2}^{2}+\frac{h^{2}}{12}\|\delta_{y}^{2}w(t)\|_{2}^{2}.
       \end{equation*}
        We recall that the operators: $\delta_{z}$, $\delta_{z}^{2}$, $\delta_{z}^{4}$ and $\delta_{2z}^{4}$, for $z\in\{x,y\}$, are defined in relation $(\ref{6})$. This completes the proof of Lemma $\ref{l2}$.
        \end{proof}

      \begin{lemma}\label{l3}
       Suppose $\{d_{n}\}_{n}$ and $\{q_{n}\}_{n}$ be sequences of nonnegative real numbers. Let $c_{0}\geq0$, be a constant and $l_{0}$ be a nonnegative rational number which is an integer multiple of $2^{-1}$, that is, $l_{0}=2^{-1}k$, where $k$ is a nonnegative integer. For $\sigma\in\{\frac{1}{2},1\}$, if
      \begin{equation*}
        q_{n+\sigma}\leq c_{0}+\underset{l=l_{0}}{\overset{n+\sigma}\sum}d_{l}q_{l},
      \end{equation*}
       then
      \begin{equation}\label{el3}
        q_{n+\sigma}\leq c_{0}\exp\left(\underset{l=l_{0}}{\overset{n+\sigma}\sum}d_{l}\right),
       \end{equation}
        where the summation index $"l"$ varies in the range: $l=l_{0},l_{0}+\frac{1}{2},l_{0}+1,...,n+\sigma-\frac{1}{2},n+\sigma$.
       \end{lemma}

    \begin{proof}
     The proof of this Lemma is similar to the well known proof of the Gronwall Lemma discussed in the literature by replacing the integer summation index  with the ones having a step size $2^{-1}$.
    \end{proof}

     With Lemmas $\ref{l1}$-$\ref{l3}$, we are ready to prove the main result of this paper (namely Theorem $\ref{t}$).

    \begin{proof} (of Theorem $\ref{t}$).\\
     Subtracting $(\ref{11})$ and $(\ref{16})$ from $(\ref{11a})$ and $(\ref{17})$, respectively, and utilizing the fact that $h_{x}=h_{y}=h$, simple calculations result in
      \begin{equation*}
       \left(\mathcal{I}-\delta_{2x}^{4}\right)(u^{n+\frac{1}{2}}_{ij}-U^{n+\frac{1}{2}}_{ij})=\left(\mathcal{I}-\delta_{2x}^{4}\right)
       (u^{n-\frac{1}{2}}_{ij}-U^{n-\frac{1}{2}}_{ij})+k\left(\gamma\delta_{2x}^{4}-\beta\delta_{x}^{4}\right)(u^{n}_{ij}-U^{n}_{ij})+
       \end{equation*}
       \begin{equation}\label{45}
       k\left[f_{1}\left(x_{i},y_{j},t_{n},u_{ij}^{n},\delta_{x}^{4}u_{ij}^{n}\right)-
       f_{1}\left(x_{i},y_{j},t_{n},U_{ij}^{n},\delta_{x}^{4}U_{ij}^{n}\right)\right]+O(k^{3}+h^{4}+kh^{4}).
       \end{equation}
      \begin{equation*}
       \left(\mathcal{I}-\alpha\delta_{2y}^{4}\right)(u^{n+1}_{ij}-U^{n+1}_{ij})-\frac{k}{4}\left(\gamma\delta_{2y}^{4}-\beta\delta_{y}^{4}\right)
       (u^{n+1}_{ij}-U^{n+1}_{ij})=\left(\mathcal{I}-\alpha\delta_{2y}^{4}\right)(u^{n+\frac{1}{2}}_{ij}-U^{n+\frac{1}{2}}_{ij})-
     \end{equation*}
     \begin{equation*}
      \frac{k}{4}\left(\gamma\delta_{2y}^{4}-\beta\delta_{y}^{4}\right)(u^{n+\frac{1}{2}}_{ij}-U^{n+\frac{1}{2}}_{ij}) +4^{-1}k\left[f_{2}(x_{i},y_{j},t_{n+1},u_{ij}^{n+1},\delta_{y}^{4}u_{ij}^{n+1})
     -f_{2}(x_{i},y_{j},t_{n+1},U_{ij}^{n+1},\delta_{y}^{4}U_{ij}^{n+1})\right]+
     \end{equation*}
     \begin{equation}\label{46}
     4^{-1}k\left[f_{2}(x_{i},y_{j},t_{n+\frac{1}{2}},u_{ij}^{n+\frac{1}{2}},\delta_{y}^{4}u_{ij}^{n+\frac{1}{2}})-
     f_{2}(x_{i},y_{j},t_{n+\frac{1}{2}},U_{ij}^{n+\frac{1}{2}},\delta_{y}^{4}U_{ij}^{n+\frac{1}{2}})\right]+O(k^{3}+h^{4}+kh^{4}).
     \end{equation}
     Let $e^{n+\frac{1}{2}}=u^{n+\frac{1}{2}}-U^{n+\frac{1}{2}}$, be the temporary error obtained at the time level $n+\frac{1}{2}$ and let $e^{n+1}=u^{n+1}-U^{n+1}$, be the error provided by the proposed formulation $(\ref{s1})$-$(\ref{s5})$ at the discrete time $t_{n+1}$. This fact combined with equations $(\ref{45})$ and $(\ref{46})$, after rearranging terms give
      \begin{equation*}
       \left(\mathcal{I}-\delta_{2x}^{4}\right)e^{n+\frac{1}{2}}_{ij}=\left(\mathcal{I}-\delta_{2x}^{4}\right)
       e^{n-\frac{1}{2}}_{ij}+k\left(\gamma\delta_{2x}^{4}-\beta\delta_{x}^{4}\right)e^{n}_{ij}+
       k\left[f_{1}\left(x_{i},y_{j},t_{n},u_{ij}^{n},\delta_{x}^{4}u_{ij}^{n}\right)-\right.
       \end{equation*}
       \begin{equation}\label{47}
       \left.f_{1}\left(x_{i},y_{j},t_{n},U_{ij}^{n},\delta_{x}^{4}U_{ij}^{n}\right)\right]+O(k^{3}+h^{4}+kh^{4}),
       \end{equation}
      \begin{equation*}
       \left(\mathcal{I}-\alpha\delta_{2y}^{4}\right)e^{n+1}_{ij}-\frac{k}{4}\left(\gamma\delta_{2y}^{4}-\beta\delta_{y}^{4}\right)
       e^{n+1}_{ij}=\left(\mathcal{I}-\alpha\delta_{2y}^{4}\right)e^{n+\frac{1}{2}}_{ij}-\frac{k}{4}\left(\gamma\delta_{2y}^{4}-
       \beta\delta_{y}^{4}\right)e^{n+\frac{1}{2}}_{ij}+
     \end{equation*}
     \begin{equation*}
      4^{-1}k\left[f_{2}(x_{i},y_{j},t_{n+1},u_{ij}^{n+1},\delta_{y}^{4}u_{ij}^{n+1})-
      f_{2}(x_{i},y_{j},t_{n+1},U_{ij}^{n+1},\delta_{y}^{4}U_{ij}^{n+1})\right]+
     \end{equation*}
     \begin{equation}\label{48}
     4^{-1}k\left[f_{2}(x_{i},y_{j},t_{n+\frac{1}{2}},u_{ij}^{n+\frac{1}{2}},\delta_{y}^{4}u_{ij}^{n+\frac{1}{2}})-
     f_{2}(x_{i},y_{j},t_{n+\frac{1}{2}},U_{ij}^{n+\frac{1}{2}},\delta_{y}^{4}U_{ij}^{n+\frac{1}{2}})\right]+O(k^{3}+h^{4}+kh^{4}).
     \end{equation}
      Multiplying both sides of equations $(\ref{47})$ and $(\ref{48})$ by $h^{2}e^{n+\frac{1}{2}}_{ij}$ and $h^{2}e^{n+1}_{ij}$, respectively, summing up from $i,j=2,3,...,M-2$, and utilizing the definition of the scalar product and the $L^{2}$-norm defined in relations $(\ref{sp})$ and $(\ref{dn})$, respectively, we obtain
      \begin{equation*}
     \left(e^{n+\frac{1}{2}}-e^{n-\frac{1}{2}},e^{n+\frac{1}{2}}\right)_{2}+\alpha\left(-\delta_{2x}^{4}e^{n+\frac{1}{2}},e^{n+\frac{1}{2}}\right)_{2}=
     -\alpha\left(\delta_{2x}^{4}e^{n-\frac{1}{2}},e^{n+\frac{1}{2}}\right)_{2}+k\left[\gamma\left(\delta_{2x}^{4}e^{n},e^{n+\frac{1}{2}}\right)_{2}-
     \beta\left(\delta_{x}^{4}e^{n},e^{n+\frac{1}{2}}\right)_{2}\right]
     \end{equation*}
      \begin{equation}\label{49}
     +kh^{2}\underset{i=2}{\overset{M-2}\sum}\underset{j=2}{\overset{M-2}\sum}\left[f_{1}\left(x_{i},y_{j},t_{n},u_{ij}^{n},\delta_{x}^{4}u_{ij}^{n}\right)-
     f_{1}\left(x_{i},y_{j},t_{n},U_{ij}^{n},\delta_{x}^{4}U_{ij}^{n}\right)\right]e^{n+\frac{1}{2}}_{ij}+\left(\mathcal{O}(k^{3}+h^{4}+kh^{4}),
     e^{n+\frac{1}{2}}\right)_{2},
     \end{equation}
      \begin{equation*}
     \left(e^{n+1}-e^{n+\frac{1}{2}},e^{n+1}\right)_{2}+\alpha\left(-\delta_{2y}^{4}e^{n+1},e^{n+1}\right)_{2}+\frac{k}{4}\left[\gamma
     \left(-\delta_{2y}^{4}e^{n+1},e^{n+1}\right)_{2}+\beta\left(\delta_{y}^{4}e^{n+1},e^{n+1}\right)_{2}\right]=
     \alpha\left(-\delta_{2y}^{4}e^{n+\frac{1}{2}},e^{n+1}\right)_{2}+
     \end{equation*}
      \begin{equation*}
      \frac{k}{4}\left[\gamma\left(-\delta_{2y}^{4}e^{n+\frac{1}{2}},e^{n+1}\right)_{2}+
     \beta\left(\delta_{y}^{4}e^{n+\frac{1}{2}},e^{n+1}\right)_{2}\right]
     +\frac{kh^{2}}{4}\underset{i=2}{\overset{M-2}\sum}\underset{j=2}{\overset{M-2}\sum}\left\{\left[f_{2}\left(x_{i},y_{j},t_{n+1},u_{ij}^{n+1},
     \delta_{y}^{4}u_{ij}^{n+1}\right)-\right.\right.
     \end{equation*}
      \begin{equation*}
     \left. f_{2}\left(x_{i},y_{j},t_{n+1},U_{ij}^{n+1},\delta_{x}^{4}U_{ij}^{n+1}\right)\right]e^{n+1}_{ij}+
     \left[f_{2}\left(x_{i},y_{j},t_{n+\frac{1}{2}},u_{ij}^{n+\frac{1}{2}}, \delta_{y}^{4}u_{ij}^{n+\frac{1}{2}}\right)-\right.
     \end{equation*}
      \begin{equation}\label{50}
     \left.\left.f_{2}\left(x_{i},y_{j},t_{n+\frac{1}{2}},U_{ij}^{n+\frac{1}{2}},\delta_{x}^{4}U_{ij}^{n+\frac{1}{2}}\right)\right]e^{n+1}_{ij}\right\}
     +\left(\mathcal{O}(k^{3}+h^{4}+kh^{4}),e^{n+\frac{1}{2}}\right)_{2},
     \end{equation}
     where $\mathcal{O}(k^{3}+h^{4}+kh^{4})=(O(k^{3}+h^{4}+kh^{4}),...,O(k^{3}+h^{4}+kh^{4}))$. Since the nonlinear functions $f_{l}$, $l=1,2$, satisfy the Lipschitz condition with respect to its fourth variable $"u"$, there are positive constants $C_{1L}$ and $C_{2L}$, which do not depend on the time step $k$ and the space step $h$, so that
      \begin{equation}\label{51}
       \left|f_{l}\left(x_{i},y_{j},t_{q},u_{ij}^{q},\delta_{z}^{4}u_{ij}^{q}\right)-
       f_{l}\left(x_{i},y_{j},t_{q},U_{ij}^{q},\delta_{z}^{4}U_{ij}^{q}\right)\right|\leq C_{lL}|u_{ij}^{q}-U_{ij}^{q}|=C_{lL}|e_{ij}^{q}|,
     \end{equation}
     for $z\in\{x,y\}$, $l=1,2$, and $q\in\{n,n+\frac{1}{2},n+1\}$. Plugging estimates $(\ref{49})$-$(\ref{51})$ together with estimates $(\ref{22})$ and $(\ref{22a})$ of Lemma $\ref{l1}$, equations $(\ref{44})$-$(\ref{45})$ of Lemma $\ref{l2}$ and multiplying the resulting relations by $2$, this provides
     \begin{equation*}
     2\left(e^{n+\frac{1}{2}}-e^{n-\frac{1}{2}},e^{n+\frac{1}{2}}\right)_{2}+2\alpha\|\delta_{x}e^{n+\frac{1}{2}}\|_{2}^{2}+\frac{\alpha h^{2}}{6}\|\delta_{x}^{2}e^{n+\frac{1}{2}}\|_{2}^{2}\leq \alpha\left[\|\delta_{x}e^{n-\frac{1}{2}}\|_{2}^{2}+
     \|\delta_{x}e^{n+\frac{1}{2}}\|_{2}^{2}+\frac{h^{2}}{6}\left(\|\delta_{x}^{2}e^{n-\frac{1}{2}}\|_{2}^{2}+\right.\right.
     \end{equation*}
      \begin{equation}\label{52}
      \left.\left.\|\delta_{x}^{2}e^{n+\frac{1}{2}}\|_{2}^{2}\right)\right]+2C_{3}k(\beta+\gamma)\|\delta_{x}e^{n}\|_{2}\|e^{n+\frac{1}{2}}\|_{2}
     +2C_{1L}kh^{2}\underset{i=2}{\overset{M-2}\sum}\underset{j=2}{\overset{M-2}\sum}|e^{n}_{ij}||e^{n+\frac{1}{2}}_{ij}|+
     \left(\mathcal{O}(k^{3}+h^{4}+kh^{4}),e^{n+\frac{1}{2}}\right)_{2},
     \end{equation}
     \begin{equation*}
     2\left(e^{n+1}-e^{n+\frac{1}{2}},e^{n+1}\right)_{2}+2\alpha\left[\|\delta_{y}e^{n+1}\|_{2}^{2}+\frac{ h^{2}}{12}\|\delta_{y}^{2}e^{n+1}\|_{2}^{2}\right]+\frac{\gamma k}{4}\left[2\|\delta_{y}e^{n+1}\|_{2}^{2}+\frac{ h^{2}}{6}\|\delta_{y}^{2}e^{n+1}\|_{2}^{2}\right] \leq
     \end{equation*}
     \begin{equation*}
      \alpha\left[\|\delta_{y}e^{n+\frac{1}{2}}\|_{2}^{2}+\|\delta_{y}e^{n+1}\|_{2}^{2}+\frac{h^{2}}{12}\left(\|\delta_{y}^{2}e^{n+\frac{1}{2}}\|_{2}^{2}
     +\|\delta_{y}^{2}e^{n+1}\|_{2}^{2}\right)\right]+\frac{C_{3}k(\beta+\gamma)}{2}\|\delta_{y}e^{n+\frac{1}{2}}\|_{2}\|e^{n+1}\|_{2}+
     \end{equation*}
      \begin{equation}\label{53}
      \frac{C_{2L}kh^{2}}{2}\underset{i=2}{\overset{M-2}\sum}\underset{j=2}{\overset{M-2}\sum}\left((e^{n+1})^{2}+|e^{n+\frac{1}{2}}|
     |e^{n+1}_{ij}|\right)+\left(\mathcal{O}(k^{3}+h^{4}+kh^{4}),e^{n+1}\right)_{2}.
     \end{equation}
      Since $\delta_{x}(\delta_{x}e^{q}_{i+\frac{1}{2},j})=\delta_{x}^{2}e^{q}_{i+1,j}$ and $\delta_{y}(\delta_{y}e^{q}_{i,j+\frac{1}{2}})=\delta_{x}^{2}e^{q}_{i,j+1}$, applying the identity: $2(v-w,v)_{2}=\|v\|_{2}^{2}+\|v-w\|_{2}^{2}-\|w\|_{2}^{2}$, along with the the Cauchy-Schwarz and Poincar\'{e}-Friedrich inequalities, direct calculations give
     \begin{equation}\label{54}
     2\left(e^{n+\frac{1}{2}}-e^{n-\frac{1}{2}},e^{n+\frac{1}{2}}\right)_{2}=\|e^{n+\frac{1}{2}}\|_{2}^{2}+\|e^{n+\frac{1}{2}}-e^{n-\frac{1}{2}}\|_{2}^{2}
     -\|e^{n-\frac{1}{2}}\|_{2}^{2},
      \end{equation}
      \begin{equation}\label{55}
     2C_{3}(\beta+\gamma)k\|\delta_{x}e^{n}\|_{2}\|\delta_{x}e^{n+\frac{1}{2}}\|_{2}\leq
     12k\|\delta_{x}e^{n}\|_{2}^{2}+\frac{C_{3}^{2}C_{p}(\beta+\gamma)^{2}k}{12}\|\delta_{x}e^{n+\frac{1}{2}}\|_{2}^{2},
      \end{equation}
      \begin{equation}\label{56}
     2C_{1L}kh^{2}\underset{i=2}{\overset{M-2}\sum}\underset{j=2}{\overset{M-2}\sum}|e^{n}_{ij}||e^{n+\frac{1}{2}}_{ij}|\leq 2C_{1L}k\|e^{n}\|_{2}\|e^{n+\frac{1}{2}}\|_{2}\leq 2C_{1L}C_{p}k\|\delta_{x}e^{n}\|_{2}\|\delta_{x}e^{n+\frac{1}{2}}\|_{2}\leq
       12k\|\delta_{x}e^{n}\|_{2}^{2}+\frac{C_{1L}^{2}C_{p}^{2}}{12}\|\delta_{x}e^{n+\frac{1}{2}}\|_{2}^{2}.
      \end{equation}
      The first estimate comes from the Cauchy-Schwarz inequality while the second and third ones follow from the Poincar\'{e}-Friedrich inequality and  estimate $2ab\leq a^{2}+b^{2}$, respectively.
      \begin{equation*}
       \left(\mathcal{O}(k^{3}+h^{4}+kh^{4}),e^{n+\frac{1}{2}}\right)_{2}\leq 2C_{4}(k^{3}+h^{4}+kh^{4})\|e^{n+\frac{1}{2}}\|_{2}\leq 2C_{4}\sqrt{C_{p}}(k^{3}+h^{4}+kh^{4})\|\delta_{x}e^{n+\frac{1}{2}}\|_{2}\leq
      \end{equation*}
      \begin{equation}\label{57}
       \frac{C_{p}k}{12}\|\delta_{x}e^{n+\frac{1}{2}}\|_{2}^{2}+12C_{4}^{2}k(k^{2}+k^{-1}h^{4}+h^{4})^{2},
      \end{equation}
      where $C_{4}$ is a positive constant independent of $h$ and $k$.
     \begin{equation}\label{58}
      2\left(e^{n+1}-e^{n+\frac{1}{2}},e^{n+1}\right)_{2}=\|e^{n+1}\|_{2}^{2}+\|e^{n+1}-e^{n+\frac{1}{2}}\|_{2}^{2}-\|e^{n+\frac{1}{2}}\|_{2}^{2},
     \end{equation}
     \begin{equation}\label{59}
     \frac{C_{3}(\beta+\gamma)k}{2}\|\delta_{y}e^{n+\frac{1}{2}}\|_{2}\|\delta_{y}e^{n+1}\|_{2}\leq \frac{C_{3}(\beta+\gamma)k}{4}
     \left[\|\delta_{y}e^{n+\frac{1}{2}}\|_{2}^{2}+\|\delta_{y}e^{n+1}\|_{2}^{2}\right],
      \end{equation}
      \begin{equation*}
      \frac{C_{2L}kh^{2}}{2}\underset{i=2}{\overset{M-2}\sum}\underset{j=2}{\overset{M-2}\sum}\left[(e^{n+1}_{ij})^{2}+|e^{n+\frac{1}{2}}_{ij}|
      |e^{n+1}_{ij}|\right]\leq \frac{C_{2L}k}{2}\left[\|e^{n+1}\|_{2}^{2}+\|e^{n+\frac{1}{2}}\|_{2}
      \|e^{n+1}\|_{2}\right]\leq \frac{C_{2L}C_{p}k}{2}\left[\|\delta_{y}e^{n+1}\|_{2}^{2}+\right.
      \end{equation*}
      \begin{equation}\label{60}
       \left.\|\delta_{y}e^{n+\frac{1}{2}}\|_{2}\|\delta_{y}e^{n+1}\|_{2}\right] \leq \frac{C_{2L}C_{p}k}{2}\left[3\|\delta_{y}e^{n+1}\|_{2}^{2}+
       \|\delta_{y}e^{n+\frac{1}{2}}\|_{2}^{2}\right],
      \end{equation}
      \begin{equation*}
       \left(\mathcal{O}(k^{3}+h^{4}+kh^{4}),e^{n+1}\right)_{2}\leq C_{4}(k^{3}+h^{4}+kh^{4})\|e^{n+1}\|_{2}\leq C_{4}\sqrt{C_{p}}(k^{3}+h^{4}+kh^{4})\|\delta_{y}e^{n+1}\|_{2}\leq
      \end{equation*}
      \begin{equation}\label{61}
       k\|\delta_{y}e^{n+1}\|_{2}^{2}+\frac{C_{4}^{2}C_{p}k}{4}(k^{2}+k^{-1}h^{4}+h^{4})^{2}.
      \end{equation}

     A combination of $(\ref{52})$ and $(\ref{54})$-$(\ref{57})$ results in
     \begin{equation*}
     \|e^{n+\frac{1}{2}}\|_{2}^{2}+\alpha\left(\|\delta_{x}e^{n+\frac{1}{2}}\|_{2}^{2}+\frac{h^{2}}{12}\|\delta_{x}^{2}e^{n+\frac{1}{2}}\|_{2}^{2}\right)+
     \|e^{n+\frac{1}{2}}-e^{n}\|_{2}^{2}\leq \|e^{n-\frac{1}{2}}\|_{2}^{2}+\alpha\left(\|\delta_{x}e^{n-\frac{1}{2}}\|_{2}^{2}
     +\frac{h^{2}}{12}\|\delta_{x}^{2}e^{n-\frac{1}{2}}\|_{2}^{2}\right)+
     \end{equation*}
     \begin{equation}\label{61aa}
      24k\|\delta_{x}e^{n}\|_{2}^{2}+\frac{1}{12}\left(C_{3}^{2}(\beta+\gamma)+C_{1L}^{2}C_{p}^{2}+C_{p}\right)k
      \|\delta_{x}e^{n+\frac{1}{2}}\|_{2}^{2}+12C_{4}^{2}k(k^{2}+k^{-1}h^{4}+h^{4})^{2}.
     \end{equation}
     Setting
     \begin{equation}\label{61a}
     E_{(z)}^{q}=\|e^{q}\|_{2}^{2}+\alpha\left(\|\delta_{z}e^{q}\|_{2}^{2}+\frac{h^{2}}{12}\|\delta_{z}^{2}e^{q}\|_{2}^{2}\right),
     \end{equation}
     where $z=x,y$. Since $0<\alpha\leq1$, for $z=x$, a combination of $(\ref{61a})$ and $(\ref{61aa})$ implies
     \begin{equation}\label{62}
      E_{(x)}^{n+\frac{1}{2}}+\|e^{n+\frac{1}{2}}-e^{n}\|_{2}^{2}\leq E_{(x)}^{n-\frac{1}{2}}+C_{\alpha}k\left(E_{(x)}^{n}+E_{(x)}^{n+\frac{1}{2}}\right)+12C_{4}^{2}k(k^{2}+k^{-1}h^{4}+h^{4})^{2},
     \end{equation}
     where $C_{\alpha}=\alpha^{-1}\max\{24,\frac{1}{12}[C_{3}^{2}(\beta+\gamma)+C_{1L}^{2}C_{p}^{2}+C_{p}]C_{p}\}$. But estimate $(\ref{62})$ holds for all $1\leq p\leq n$. Summing up this estimate from $p=1,2,...,n$, provides
     \begin{equation*}
      E_{(x)}^{n+\frac{1}{2}}+\underset{p=1}{\overset{n}\sum}\|e^{p+\frac{1}{2}}-e^{p}\|_{2}^{2} \leq E_{(x)}^{\frac{1}{2}}+C_{\alpha}k\underset{p=1}{\overset{n}\sum}\left(E_{(x)}^{p}+E_{(x)}^{p+\frac{1}{2}}\right)+12C_{4}^{2}
      kn(k^{2}+k^{-1}h^{4}+h^{4})^{2}.
     \end{equation*}
     Since $kn\leq kN=T$, introducing a summation index $l$ which varies in the range: $l=1,\frac{3}{2},2,...,n,n+\frac{1}{2}$, it is not difficult to observe that this estimate implies
     \begin{equation*}
      E_{(x)}^{n+\frac{1}{2}}\leq E_{(x)}^{\frac{1}{2}}+C_{\alpha}k\underset{l=1}{\overset{n+\frac{1}{2}}\sum}E_{(x)}^{l}+12C_{4}^{2}
      T(k^{2}+k^{-1}h^{4}+h^{4})^{2}.
     \end{equation*}
     Utilizing estimate $(\ref{el3})$ given in Lemma $\ref{l3}$, this becomes
     \begin{equation*}
      E_{(x)}^{n+\frac{1}{2}}\leq \left(E_{(x)}^{\frac{1}{2}}+12C_{4}^{2}T(k^{2}+k^{-1}h^{4}+h^{4})^{2}\right)\exp\left(\underset{l=1}
      {\overset{n+\frac{1}{2}}\sum}C_{\alpha}k\right)\leq \left(E_{(x)}^{\frac{1}{2}}+12C_{4}^{2}T(k^{2}+k^{-1}h^{4}+h^{4})^{2}\right)
      \exp\left(2nC_{\alpha}k\right).
     \end{equation*}
     Because $n=1,2,...,N-1$, and $k=\frac{T}{N}$, so $2nk\leq 2Nk=2T$. Thus
     \begin{equation}\label{63}
      E_{(x)}^{n+\frac{1}{2}}\leq e^{2C_{\alpha}T}\left(E_{(x)}^{\frac{1}{2}}+12C_{4}^{2}T(k^{2}+k^{-1}h^{4}+h^{4})^{2}\right).
     \end{equation}
     Now, subtracting equation $(\ref{s1})$ from $(\ref{18})$ and using the initial condition $(\ref{s5})$, this results in
      \begin{equation*}
      \left(\mathcal{I}-\delta_{2x}^{4}\right)e^{\frac{1}{2}}_{ij}-\frac{k}{4}\left(\gamma\delta_{2x}^{4}-\beta\delta_{x}^{4}\right)e^{\frac{1}{2}}_{ij}
      =\frac{k}{4}\left[f_{1}(x_{i},y_{j},t_{\frac{1}{2}},u_{ij}^{\frac{1}{2}},\delta_{x}^{4}u_{ij}^{\frac{1}{2}})-f_{1}(x_{i},y_{j},t_{\frac{1}{2}},
      U_{ij}^{\frac{1}{2}},\delta_{x}^{4}U_{ij}^{\frac{1}{2}})\right]+O(k^{3}+h^{4}+kh^{4}).
     \end{equation*}
     Multiplying side by side this equation by $2h^{2}e^{\frac{1}{2}}_{ij}$, summing up from $i,j=2,3,...,M-2$, utilizing Lemma $\ref{l2}$ together with estimate $(\ref{51})$ and rearranging terms to obtain
     \begin{equation}\label{64}
      \|e^{\frac{1}{2}}\|_{2}^{2}+\alpha\left(\|\delta_{x}e^{\frac{1}{2}}\|_{2}^{2}+\frac{h^{2}}{12}\|\delta_{x}^{2}e^{\frac{1}{2}}\|_{2}^{2}\right)
      +\frac{\gamma k}{4}\left(\|\delta_{x}e^{\frac{1}{2}}\|_{2}^{2}+\frac{h^{2}}{12}\|\delta_{x}^{2}e^{\frac{1}{2}}\|_{2}^{2}\right)\leq\frac{C_{1L}k}{4}
      \|e^{\frac{1}{2}}\|_{2}^{2}+C_{5}k(k^{2}+k^{-1}h^{4}+h^{4})\|e^{\frac{1}{2}}\|_{2},
     \end{equation}
     where $C_{5}>0$ is a constant that does not depend on the time step $k$ and mesh size $h$. But it is easy to see that
     \begin{equation*}
     C_{5}k(k^{2}+k^{-1}h^{4}+h^{4})\|e^{\frac{1}{2}}\|_{2}\leq \frac{k}{4}\|e^{\frac{1}{2}}\|_{2}^{2}+C_{5}^{2}k(k^{2}+k^{-1}h^{4}+h^{4})^{2}.
     \end{equation*}
     Substituting this into estimate $(\ref{64})$ yields
     \begin{equation*}
     \|e^{\frac{1}{2}}\|_{2}^{2}+\alpha\left(\|\delta_{x}e^{\frac{1}{2}}\|_{2}^{2}+\frac{h^{2}}{12}\|\delta_{x}^{2}e^{\frac{1}{2}}\|_{2}^{2}\right)
      +\frac{\gamma k}{4}\left(\|\delta_{x}e^{\frac{1}{2}}\|_{2}^{2}+\frac{h^{2}}{12}\|\delta_{x}^{2}e^{\frac{1}{2}}\|_{2}^{2}\right)
      \leq\frac{(1+C_{1L})k}{4}\|e^{\frac{1}{2}}\|_{2}^{2}+C_{5}^{2}k(k^{2}+k^{-1}h^{4}+h^{4})^{2},
     \end{equation*}
     which implies
     \begin{equation*}
     E_{(x)}^{\frac{1}{2}} \leq\frac{(1+C_{1L})k}{4}E_{(x)}^{\frac{1}{2}}+C_{5}^{2}k(k^{2}+k^{-1}h^{4}+h^{4})^{2}.
     \end{equation*}
     Apply estimate $(\ref{el3})$ of Lemma $\ref{l3}$ to get
     \begin{equation*}
     E_{(x)}^{\frac{1}{2}} \leq C_{5}^{2}k(k^{2}+k^{-1}h^{4}+h^{4})^{2}\exp\left(\frac{(1+C_{1L})k}{4}\right).
     \end{equation*}
     Since $0<k\leq1$, this becomes
     \begin{equation}\label{65}
     E_{(x)}^{\frac{1}{2}} \leq C_{5}^{2}\exp\left(\frac{(1+C_{1L})}{4}\right)(k^{2}+k^{-1}h^{4}+h^{4})^{2}.
     \end{equation}
     A combination of relations $(\ref{65})$ and $(\ref{63})$ gives
     \begin{equation}\label{66}
      E_{(x)}^{n+\frac{1}{2}}\leq e^{2C_{\alpha}T}\left(C_{5}^{2}\exp\left(\frac{(1+C_{1L})}{4}\right)+12C_{4}^{2}T\right)(k^{2}+k^{-1}h^{4}+h^{4})^{2}.
     \end{equation}
     Plugging estimates $(\ref{53})$ and $(\ref{58})$-$(\ref{61})$, straightforward computations provide
     \begin{equation*}
     \|e^{n+1}\|_{2}^{2}+2\alpha\left(\|\delta_{y}e^{n+1}\|_{2}^{2}+\frac{h^{2}}{12}\|\delta_{y}^{2}e^{n+1}\|_{2}^{2}\right)+
     \frac{\gamma k}{2}\left(\|\delta_{y}e^{n+1}\|_{2}^{2}+\frac{h^{2}}{12}\|\delta_{y}^{2}e^{n+1}\|_{2}^{2}\right)+
     \|e^{n+1}-e^{n+\frac{1}{2}}\|_{2}^{2}\leq \|e^{n+\frac{1}{2}}\|_{2}^{2}+
     \end{equation*}
     \begin{equation*}
     \alpha\left[\|\delta_{y}e^{n+\frac{1}{2}}\|_{2}^{2}+\|\delta_{y}^{2}e^{n+1}\|_{2}^{2}+\frac{h^{2}}{12}
     \left(\|\delta_{y}^{2}e^{n+\frac{1}{2}}\|_{2}^{2}+\|\delta_{y}^{2}e^{n+1}\|_{2}^{2}\right)\right]+
      4^{-1}k(C_{2L}C_{p}+\gamma+\beta)\|\delta_{y}e^{n+\frac{1}{2}}\|_{2}^{2}+
     \end{equation*}
      \begin{equation*}
      4^{-1}k(1+\beta+\gamma+3C_{2L}C_{p})\|\delta_{y}e^{n+1}\|_{2}^{2}+4^{-1}C_{4}^{2}C_{p}^{2}k(k^{2}+k^{-1}h^{4}+h^{4})^{2}.
     \end{equation*}
      Utilizing relation $(\ref{61a})$ with $z=y$, it is not hard to observe that this estimate implies
     \begin{equation}\label{67}
     E_{(y)}^{n+1}\leq E_{(y)}^{n+\frac{1}{2}}+4^{-1}k(1+\beta+\gamma+3C_{2L}C_{p})
     \left(\|\delta_{y}e^{n+\frac{1}{2}}\|_{2}^{2}+\|\delta_{y}e^{n+1}\|_{2}^{2}\right)+4^{-1}C_{4}^{2}C_{p}^{2}k(k^{2}+k^{-1}h^{4}+h^{4})^{2}.
     \end{equation}
     Replacing the discrete times $t_{n+\frac{1}{2}}$ and $t_{n+1}$ in the predictor and corrector phases with $t_{n}$ and $t_{n+\frac{1}{2}}$, respectively, using Lemmas $\ref{l1}$-$\ref{l3}$ and performing straightforward calculations, this results in
     \begin{equation}\label{68a}
      E_{(x)}^{n}\leq E_{(x)}^{n-1}+C_{\alpha}k\left(E_{(x)}^{n-\frac{1}{2}}+E_{(x)}^{n}\right)+12C_{4}^{2}k(k^{2}+k^{-1}h^{4}+h^{4})^{2},
     \end{equation}
     \begin{equation}\label{68}
     E_{(y)}^{n+\frac{1}{2}}\leq E_{(y)}^{n}+4^{-1}k(1+\beta+\gamma+3C_{2L}C_{p})\left(\|\delta_{y}e^{n+\frac{1}{2}}\|_{2}^{2}
     +\|\delta_{y}e^{n}\|_{2}^{2}\right)+4^{-1}C_{4}^{2}C_{p}^{2}k(k^{2}+k^{-1}h^{4}+h^{4})^{2},
     \end{equation}
      for $n\geq1$, and
     \begin{equation}\label{69}
     E_{(y)}^{\frac{1}{2}} \leq C_{6}^{2}\exp\left(\frac{1+C_{2L}}{4}\right)(k^{2}+k^{-1}h^{4}+h^{4})^{2},
     \end{equation}
     where $C_{6}$ is a positive constant independent of the space step $h$ and the time step $k$. Inequalities $(\ref{67})$ and $(\ref{68})$ indicate that
     \begin{equation*}
     E_{(y)}^{q+1}\leq E_{(y)}^{q+\frac{1}{2}}+4^{-1}k(1+\beta+\gamma+3C_{2L}C_{p})\left(\|\delta_{y}e^{q+\frac{1}{2}}\|_{2}^{2}
     +\|\delta_{y}e^{q+1}\|_{2}^{2}\right)+4^{-1}C_{4}^{2}C_{p}^{2}k(k^{2}+k^{-1}h^{4}+h^{4})^{2},
     \end{equation*}
     for $q=0,\frac{1}{2},1,...,n-\frac{1}{2},n$. Summing up this estimate for $q=0,\frac{1}{2},1,...,n$, yields
     \begin{equation*}
     E_{(y)}^{n+1}\leq E_{(y)}^{\frac{1}{2}}+4^{-1}k(1+\beta+\gamma+3C_{2L}C_{p})\underset{q=\frac{1}{2}}{\overset{n+\frac{1}{2}}\sum}
     \left(\|\delta_{y}e^{q+\frac{1}{2}}\|_{2}^{2}+\|\delta_{y}e^{q}\|_{2}^{2}\right)+4^{-1}C_{4}^{2}C_{p}^{2}(2n+1)k(k^{2}+k^{-1}h^{4}+h^{4})^{2}.
     \end{equation*}
     But $(2n+1)k\leq 2T+1$. This fact combined with this estimate and $(\ref{69})$ give
     \begin{equation*}
     E_{(y)}^{n+1}\leq \left(C_{6}^{2}e^{\frac{1+C_{2L}}{4}}+\frac{C_{4}^{2}C_{p}^{2}(2T+1)}{4}\right)
     (k^{2}+k^{-1}h^{4}+h^{4})^{2}+\frac{1+\beta+\gamma+3C_{2L}C_{p}}{4}k\underset{q=\frac{1}{2}}{\overset{n+\frac{1}{2}}\sum}
     \left(\|\delta_{y}e^{q+\frac{1}{2}}\|_{2}^{2}+\|\delta_{y}e^{q}\|_{2}^{2}\right).
     \end{equation*}
     This can be rewritten after rearranging terms as
      \begin{equation*}
     E_{(y)}^{n+1}\leq \left(C_{6}^{2}e^{\frac{1+C_{2L}}{4}}+\frac{C_{4}^{2}C_{p}^{2}(2T+1)}{4}\right)
     (k^{2}+k^{-1}h^{4}+h^{4})^{2}+\frac{1+\beta+\gamma+3C_{2L}C_{p}}{2}k\underset{q=\frac{1}{2}}{\overset{n+1}\sum}\|\delta_{y}e^{q}\|_{2}^{2}.
     \end{equation*}
     Indeed, the summation index varies in the range: $\frac{1}{2},1,...,n,n+\frac{1}{2}$. Applying inequality $(\ref{el3})$ given in Lemma $\ref{l3}$ to get
     \begin{equation*}
     E_{(y)}^{n+1}\leq \left(C_{6}^{2}e^{\frac{1+C_{2L}}{4}}+\frac{C_{4}^{2}C_{p}^{2}(2T+1)}{4}\right)
     (k^{2}+k^{-1}h^{4}+h^{4})^{2}\exp\left((1+\beta+\gamma+3C_{2L}C_{p})(n+1)k\right).
     \end{equation*}
     Since $(n+1)k\leq T+1$, absorbing all the constants into a positive constant $\widetilde{C}_{1}$, we obtain
      \begin{equation}\label{70}
     E_{(y)}^{n+1}\leq \widetilde{C}_{1}(k^{2}+k^{-1}h^{4}+h^{4})^{2}.
     \end{equation}
     In addition, it is easy to see that $(\ref{66})$ implies
     \begin{equation}\label{71}
     E_{(x)}^{n+\frac{1}{2}}\leq \widetilde{C}_{2}(k^{2}+k^{-1}h^{4}+h^{4})^{2},
     \end{equation}
     where all the constant have been absorbed into a positive constant $\widetilde{C}_{2}$. Analogously, it is not hard to establish that the term $E_{(x)}^{n}$ and $E_{(y)}^{n+\frac{1}{2}}$, associated with the predicted error $e^{n}_{ij}$ and the corrected one $e^{n+\frac{1}{2}}_{ij}$, at the discrete points $(x_{i},y_{j},t_{n})$ and $(x_{i},y_{j},t_{n+\frac{1}{2}})$, respectively, satisfy
     \begin{equation}\label{72}
     E_{(x)}^{n}\leq \widehat{C}_{1}(k^{2}+k^{-1}h^{4}+h^{4})^{2},\text{\,\,\,for\,\,\,}n\geq1,
     \end{equation}
      \begin{equation}\label{73}
     E_{(y)}^{n+\frac{1}{2}}\leq \widehat{C}_{2}(k^{2}+k^{-1}h^{4}+h^{4})^{2},\text{\,\,\,for\,\,\,}n\geq0.
     \end{equation}
      Plugging estimates $(\ref{71})$, $(\ref{73})$ and $(\ref{61a})$ (respectively, $(\ref{70})$, $(\ref{72})$ and $(\ref{61a})$), utilizing the time step requirement $(\ref{sR})$, that is $k=h^{4/3}$, along with the $L^{\infty}(0,T;H^{2})$-norm given in $(\ref{dn})$, provides
     \begin{equation*}
      \|e^{n+\sigma}\|_{H^{2}}^{2}\leq 2\|e^{n+\sigma}\|_{2}^{2}+\alpha\left[\|\delta_{x}e^{n+\sigma}\|_{2}^{2}+\|\delta_{y}e^{n+\sigma}\|_{2}^{2}+
      \frac{h^{2}}{12}\left(\|\delta_{x}^{2}e^{n+\sigma}\|_{2}^{2}+\|\delta_{y}^{2}e^{n+\sigma}\|_{2}^{2}\right)\right]= E_{(x)}^{n+\sigma}+E_{(y)}^{n+\sigma}\leq
     \end{equation*}
     \begin{equation*}
     \max\{\widetilde{C}_{1}+\widehat{C}_{1},\widetilde{C}_{2}+\widehat{C}_{2}\}(k^{2}+h^{8/3}+h^{4})^{2},
     \end{equation*}
     for $\sigma=\frac{1}{2},1$, and any $n=0,1,2,...,N-1$. The maximum over $n$, for $0\leq n\leq N-1$, on the left and right sides of these estimates implies
     \begin{equation}\label{74}
      \||e|\|_{H^{2},\infty}^{2}\leq \widetilde{C}_{3}(k^{2}+h^{8/3}+h^{4})^{2},
     \end{equation}
     where $\widetilde{C}_{3}=\max\{\widetilde{C}_{1}+\widehat{C}_{1},\widetilde{C}_{2}+\widehat{C}_{2}\}$. Taking the square root in both sides of inequality $(\ref{74})$, using relation $\||u|\|_{H^{2},\infty}-\||U|\|_{H^{2},\infty}\leq \||u-U|\|_{H^{2},\infty}=\||e|\|_{H^{2},\infty}$, together with the regularity condition $(\ref{RC})$, to get
     \begin{equation*}
     \||U|\|_{H^{2},\infty}\leq C_{0}+\widehat{C}_{3}(k^{2}+h^{8/3}+h^{4}),
     \end{equation*}
     where $\widehat{C}_{3}=\sqrt{\widetilde{C}_{3}}$. This ends the proof of estimate $(\ref{20})$ in Theorem $\ref{t}$. Finally, since $h<1$, so $h^{4}<h^{8/3}$. Substituting this into $(\ref{74})$ and taking the square root of the new estimate gives
     \begin{equation*}
      \||e|\|_{H^{2},\infty}\leq \widetilde{C}_{4}(k^{2}+h^{8/3}),
     \end{equation*}
     where $\widetilde{C}_{4}>0$, is a constant that does not depend on the time step $k$ and the grid size $h$. Thus, the proof of Theorem $\ref{t}$ is completed.
     \end{proof}

   \section{Numerical experiments and Convergence rate}\label{sec4}
    This section presents some numerical evidences to verify the stability and convergence order of the developed technique $(\ref{s1})$-$(\ref{s5})$, for solving the two-dimensional unsteady sobolev and regularized long wave equation $(\ref{1})$ subjected to initial-boundary conditions $(\ref{2})$-$(\ref{3})$. To demonstrate the efficiency of the new algorithm, we set $k=h^{\frac{4}{3}}$, for $h=2^{-l}$, $l=1,2,3,4,5$, and we compute the exact solution: $u^{n}$, approximate solution: $U^{n}$, and the error: $e^{n}=u^{n}-U^{n}$, at discrete time $t_{n}$ in the $L^{\infty}(0,T;L^{2}(\Omega))$-norm, denoted $\||\cdot|\|_{2,\infty}$ and defined in relation $(\ref{dn})$. The following formula is considered for the calculation of the error
     \begin{equation*}
       \||e(h)|\|_{2,\infty}=\underset{0\leq n\leq N}{\max}\|u^{n}-U_{h}^{n}\|_{2},
       \end{equation*}
       where $U_{h}^{n}$ denotes the numerical solution obtained at time level $n$ and associated with a mesh size $h$. Furthermore, the convergence rate $R(k,h)$ of the proposed approach is estimated using the formula
       \begin{equation*}
        R(k,h)=\log_{2}(\||e(2h)|\|_{2,\infty}/\||e(h)|\|_{2,\infty}).
       \end{equation*}
       Here we set $k=h^{\frac{4}{3}}$. Lastly, the numerical calculations are carried out using MATLAB R$2013b$.\\

   $\bullet$ \textbf{Example $1.$} Let $\Omega=(0,1)^{2}$ be the fluid region and let $T=1$ be the final time. Consider the problem given in \cite{ks} by
        \begin{equation*}
        u_{t}-(u_{txx}+u_{tyy})-(u_{xx}+u_{yy})+u=0,\text{\,\,\,on\,\,\,\,}\Omega\times(0,1], \\
       \end{equation*}
   with initial and boundary conditions
   \begin{equation*}
    \left.
      \begin{array}{ll}
        u(x,y,0)=\sin(\pi x)\sin(\pi y),\text{\,\,\,\,for\,\,\,\,\,}(x,y)\in\Omega\cup\partial\Omega, & \hbox{\,} \\
       \text{\,}\\
        u(x,y,t)=0,\text{\,\,\,\,for\,\,\,\,\,}(x,y,t)\in\partial\Omega\times[0,1]. & \hbox{,} \\
      \end{array}
    \right.
   \end{equation*}
    The exact solution $u$ of the initial-boundary value problem $(\ref{1})$-$(\ref{3})$ is given by
      \begin{equation*}
     u(x,y,t)=e^{-t}\sin(\pi x)\sin(\pi y).
     \end{equation*}

       \textbf{Table 1} $\label{T1}$. Stability and convergence rate $R(k,h)$ of the new three-level time split technique with varying mesh grid $h$ and time step $k$, where $k=h^{4/3}$.
          \begin{equation*}
           \begin{tabular}{|c|c|c|c|c|}
            \hline
            $h$ & $\||u|\|_{\infty,2}$ &$\||U|\|_{\infty,2}$ & $\||e(h)|\|_{\infty,2}$ & R(k,h)\\
             \hline
            $2^{-1}$ & $5.000\times10^{-1}$ & $4.8602\times10^{-1}$  &  $1.0113\times10^{-2}$ &   --  \\
            \hline
            $2^{-2}$ & $4.9908\times10^{-1}$ & $4.9863\times10^{-1}$  &  $1.7103\times10^{-3}$ & 2.5738 \\
            \hline
            $2^{-3}$ & $4.9998\times10^{-1}$ & $4.9887\times10^{-1}$  & $2.7873\times10^{-4}$ & 2.6173 \\
            \hline
            $2^{-4}$ & $4.9898\times10^{-1}$ & $4.9779\times10^{-1}$  &  $4.4420\times10^{-5}$  & 2.6496 \\
            \hline
            $2^{-5}$ & $5.0001\times10^{-1}$ & $5.0002\times10^{-1}$   &  $6.8298\times10^{-5}$ & 2.7013 \\
           \hline
          \end{tabular}
            \end{equation*}

      $\bullet$ \textbf{Example $2$}\cite{ks}. Let $\Omega=(0,1)\times(0,1)$ and $[0,T]=[0,1]$. Consider equation $(\ref{1})$ of the form
      \begin{equation*}
        u_{t}-(u_{txx}+u_{tyy})-(u_{xx}+u_{yy})+f(x,y,t,u)=0,\text{\,\,\,on\,\,\,\,}\Omega\times(0,1], \\
       \end{equation*}
     subjects to initial and boundary conditions
   \begin{equation*}
    \left.
      \begin{array}{ll}
        u(x,y,0)=\sin(\pi x)\sin(\pi y)e^{x+y},\text{\,\,\,\,on\,\,\,\,\,}\Omega\cup\partial\Omega, & \hbox{\,} \\
       \text{\,}\\
        u(x,y,t)=0,\text{\,\,\,\,on\,\,\,\,\,}\partial\Omega\times[0,1], & \hbox{,} \\
      \end{array}
    \right.
   \end{equation*}
     where $f(x,y,t,u)=(4\pi^{2}-3)u-4\pi e^{x+y+t}[\sin(\pi x)\cos(\pi y)+\cos(\pi x)\sin(\pi y)]$.
    The analytical solution $u$ is defined as
      \begin{equation*}
     u(x,y,t)=\sin(\pi x)\sin(\pi y)e^{x+y+t}.
     \end{equation*}

       \textbf{Table 2} $\label{T2}$. Analysis of stability and convergence rate $R(k,h)$ of the developed three-level time split Leapfrog/Crank-Nicolson approach with varying space step $h$ and time step $k$. We set $k=h^{4/3}$.
          \begin{equation*}
            \begin{tabular}{|c|c|c|c|c|}
            \hline
            $h$ & $\||u|\|_{\infty,2}$ &$\||U|\|_{\infty,2}$ & $\||e(h)|\|_{\infty,2}$ & R(k,h)\\
             \hline
            $2^{-1}$ & $3.6945\times10^{0}$ & $3.6645\times10^{0}$  &  $7.3932\times10^{-2}$ &   --  \\
            \hline
            $2^{-2}$ & $3.9303\times10^{0}$ & $3.9224\times10^{0}$  &  $1.2343\times10^{-2}$ &  2.5849  \\
            \hline
            $2^{-3}$ & $3.9417\times10^{0}$ & $3.9369\times10^{0}$  &  $2.0011\times10^{-3}$ & 2.6173 \\
            \hline
            $2^{-4}$ & $3.9423\times10^{0}$ & $3.9415\times10^{0}$  &  $3.1674\times10^{-4}$ & 2.6594 \\
            \hline
            $2^{-5}$ & $3.9431\times10^{1}$ & $3.9427\times10^{1}$  &  $4.8036\times10^{-5}$ & 2.7211 \\
           \hline
          \end{tabular}
            \end{equation*}

      $\bullet$ \textbf{Example $3$}\cite{ks}. Suppose $\Omega=(0,1)\times(0,1)$ and $T=1$. Consider the $2$D sobolev and regularized long wave equation given in \cite{ks} by
      \begin{equation*}
        u_{t}-(u_{txx}+u_{tyy})-(u_{x}+u_{y})+uu_{x}+uu_{y}=0,\text{\,\,\,on\,\,\,\,}\Omega\times(0,1], \\
       \end{equation*}
     with initial condition
     \begin{equation*}
      u(x,y,0)=\frac{1}{2}\sec h^{2}[\frac{1}{\sqrt{2}}(x+y)],\text{\,\,\,\,on\,\,\,\,\,}\Omega\cup\partial\Omega,
     \end{equation*}
      and boundary condition
    \begin{equation*}
     u(0,y,t)=\sec h^{2}(y-t),\text{\,\,\,}u(1,y,t)=\sec h^{2}(y-t+1),\text{\,\,\,}u(x,0,t)=\sec h^{2}(-x+t),\text{\,\,\,}u(x,1,t)=\sec h^{2}(-x+t-1).
     \end{equation*}
     The exact solution $u$ is given by
      \begin{equation*}
     u(x,y,t)=\sec h^{2}(x+y-t).
     \end{equation*}

       \textbf{Table 3} $\label{T3}$. Stability analysis and convergence order $R(k,h)$ of the proposed three-level time split Leapfrog/Crank-Nicolson algorithm with varying mesh size $h$ and time step $k$. We take $k=h^{4/3}$.
          \begin{equation*}
           \begin{tabular}{|c|c|c|c|c|}
            \hline
            $h$ & $\||u|\|_{\infty,2}$ &$\||U|\|_{\infty,2}$ & $\||e(h)|\|_{\infty,2}$ & R(k,h)\\
             \hline
            $2^{-1}$ & $1.2069\times10^{0}$ & $1.1828\times10^{0}$  &  $2.4124\times10^{-2}$ &   --  \\
            \hline
            $2^{-2}$ & $1.0475\times10^{0}$ & $1.0433\times10^{0}$  &  $4.2324\times10^{-3}$ &  2.5739  \\
            \hline
            $2^{-3}$ & $9.6445\times10^{-1}$ & $9.6403\times10^{-1}$  & $6.9341\times10^{-4}$  & 2.6097 \\
            \hline
            $2^{-4}$ & $9.2189\times10^{-1}$ & $9.2197\times10^{-1}$  &  $1.1183\times10^{-4}$ & 2.6324 \\
            \hline
            $2^{-5}$ & $9.2188\times10^{-1}$ & $9.2188\times10^{-1}$  &  $1.7575\times10^{-5}$ & 2.6697 \\
           \hline
          \end{tabular}
            \end{equation*}
        Figures $\ref{figure1}$-$\ref{figure3}$ indicate that the developed two-step explicit/implicit approach $(\ref{s1})$-$(\ref{s5})$ is stable whereas
        \textbf{Tables} $1$-$3$ suggest that the new algorithm is temporal second-order convergent and spatial accurate with order $O(k^{8/3})$. These numerical results confirm the theory (see Section $\ref{sec3}$, Theorem $\ref{t}$).\\

     \section{General conclusions and future works}\label{sec5}
     This paper has developed a three-level time split high-order Leapfrog/Crank-Nicolson approach in an approximate solution of the two-dimensional time-dependent sobolev and regularized long wave equations arising in fluid mechanics. The stability and error estimates of the proposed numerical technique have been deeply analyzed in $L^{\infty}(0,T;H^{2})$-norm. The theoretical studies have shown that the developed numerical scheme is stable and accurate with order $O(k^{2}+h^{8/3})$, where $k$ and $h$ represent the time step and space step, respectively. These results have been confirmed by three numerical examples. Furthermore, both theoretical and numerical results suggest that the proposed algorithm $(\ref{s1})$-$(\ref{s5})$ is computational less expensive, faster and more efficient than a large class of numerical methods widely discussed in the literature \cite{1ks,ks,6ks,7ks,13ks,11ks,16ks} for solving the initial-boundary value problem $(\ref{1})$-$(\ref{3})$. The use of the $L^{\infty}(0,T;H^{2})$-norm indicates that the constructed scheme $(\ref{s1})$-$(\ref{s5})$ can be considered as a strong numerical method for the integration of general systems of $(2+1)$-dimensional nonlinear sobolev problems. Our future works will develop a three-level time split high-order Leapfrog/Crank-Nicolson formulation for solving the generalized $(2+1)$-dimensional nonlinear evolutionary models. Both stability and convergence rate of the constructed technique also will be analyzed in $L^{\infty}(0,T;H^{2})$-norm.\\
     \text{\,}\\
     \textbf{Acknowledgment.} This work has been partially supported by the Deanship of Scientific Research of Imam Mohammad Ibn Saud Islamic University
    (IMSIU) under the Grant No. $331203.$\\

    \newpage

          \begin{figure}
         \begin{center}
        Stability and convergence of the constructed three-level time split Leapfrog/Crank-Nicolson with $k=h^{4/3}.$
         \begin{tabular}{c c}
         \psfig{file=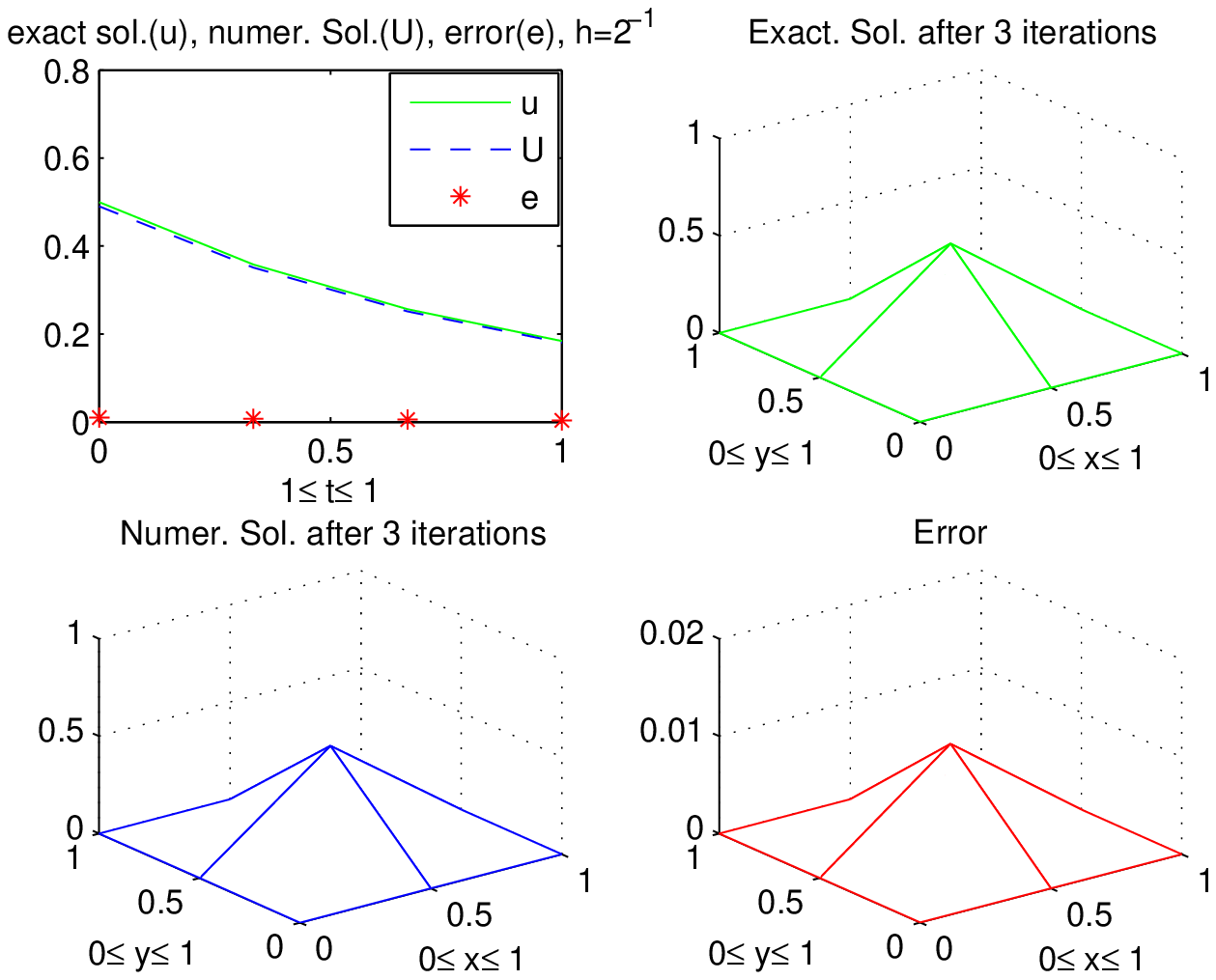,width=7cm} & \psfig{file=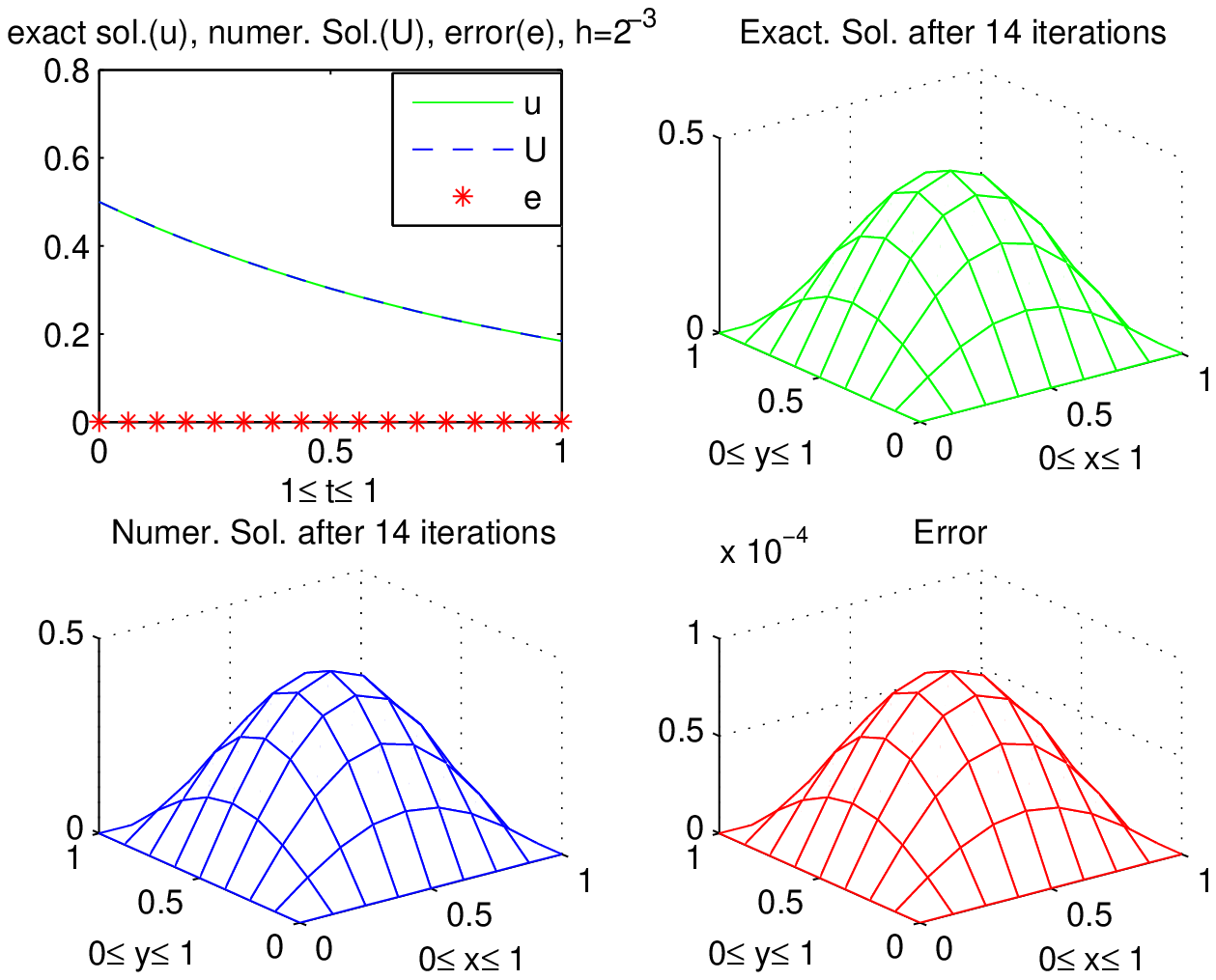,width=7cm}\\
         \psfig{file=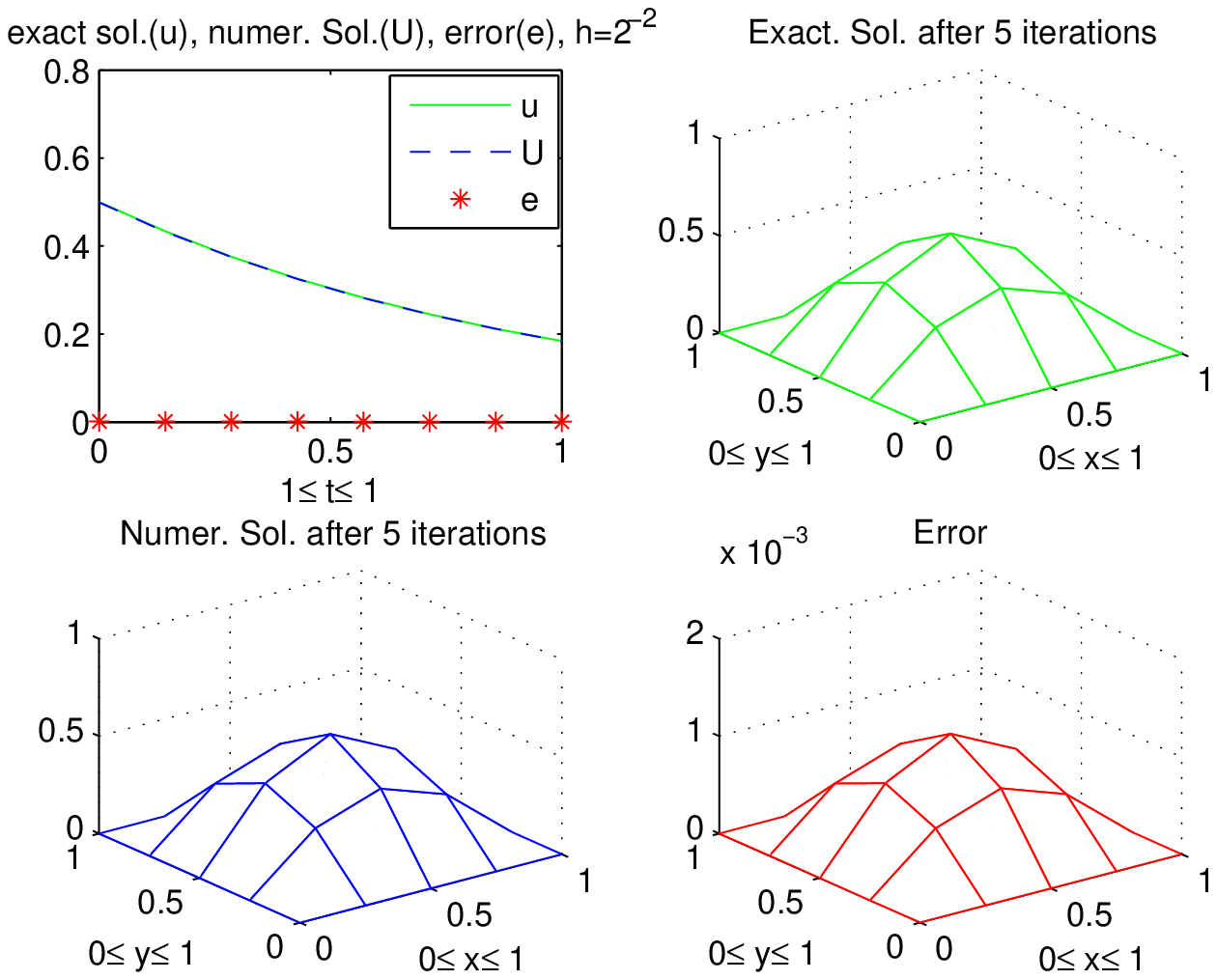,width=7cm} & \psfig{file=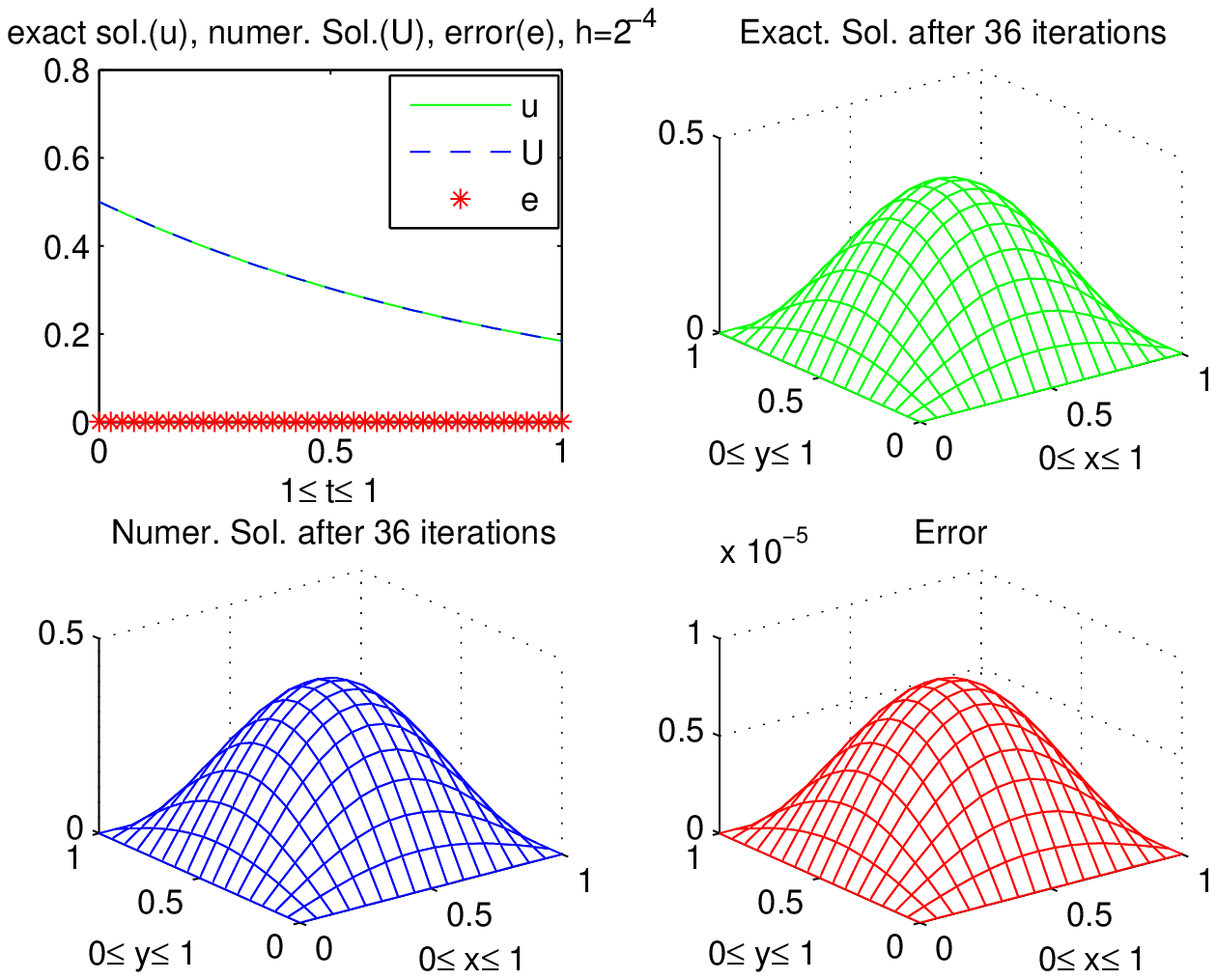,width=7cm}\\
         \end{tabular}
        \end{center}
        \caption{Exact solution(u: in green), numerical solution (U: in blue) and error(E: in red) for Example 1}
        \label{figure1}
        \end{figure}

           \begin{figure}
         \begin{center}
         Stability analysis and convergence of the developed three-level time split Leapfrog/Crank-Nicolson approach with $k=h^{4/3}$.
         \begin{tabular}{c c}
         \psfig{file=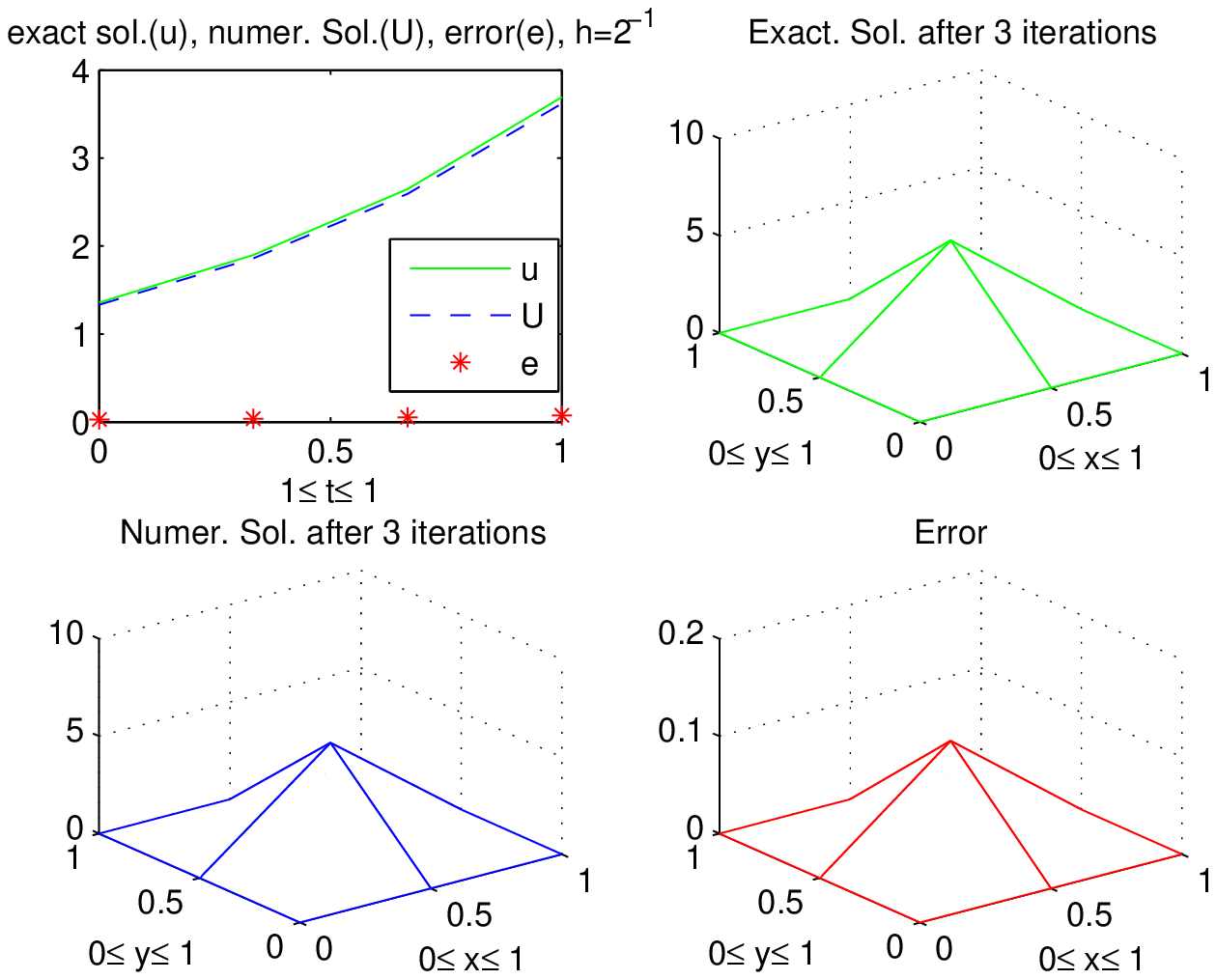,width=7cm} & \psfig{file=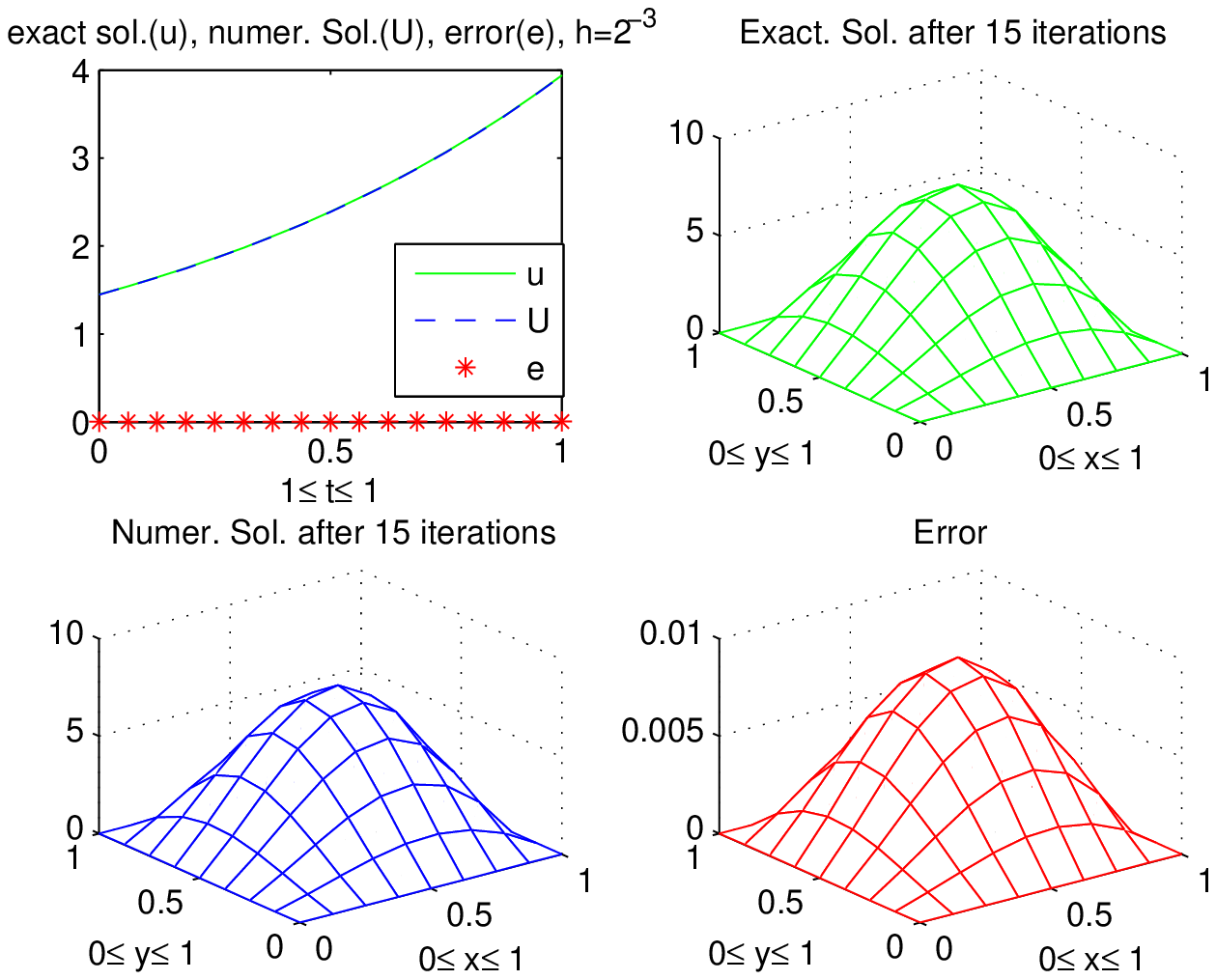,width=7cm}\\
         \psfig{file=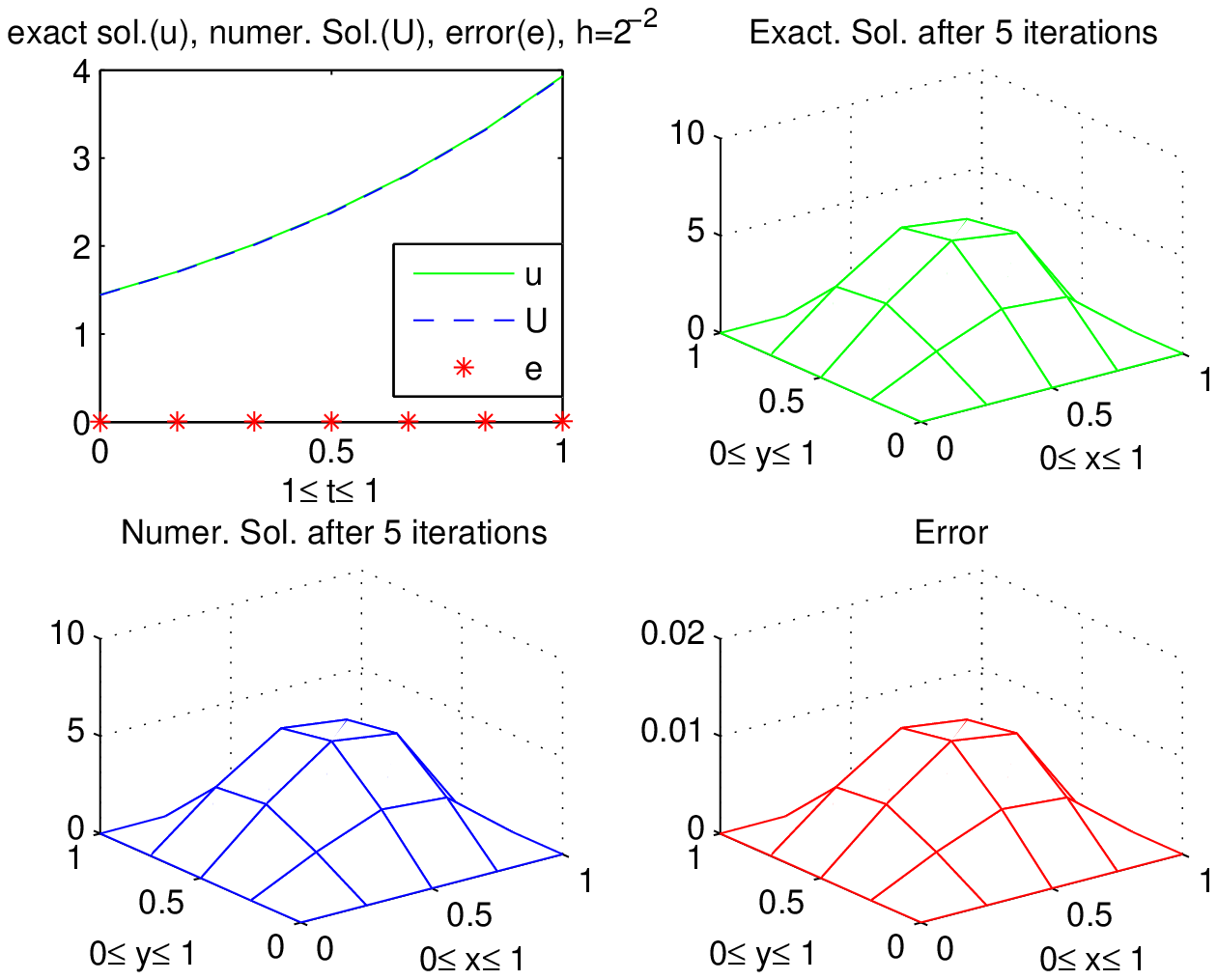,width=7cm} & \psfig{file=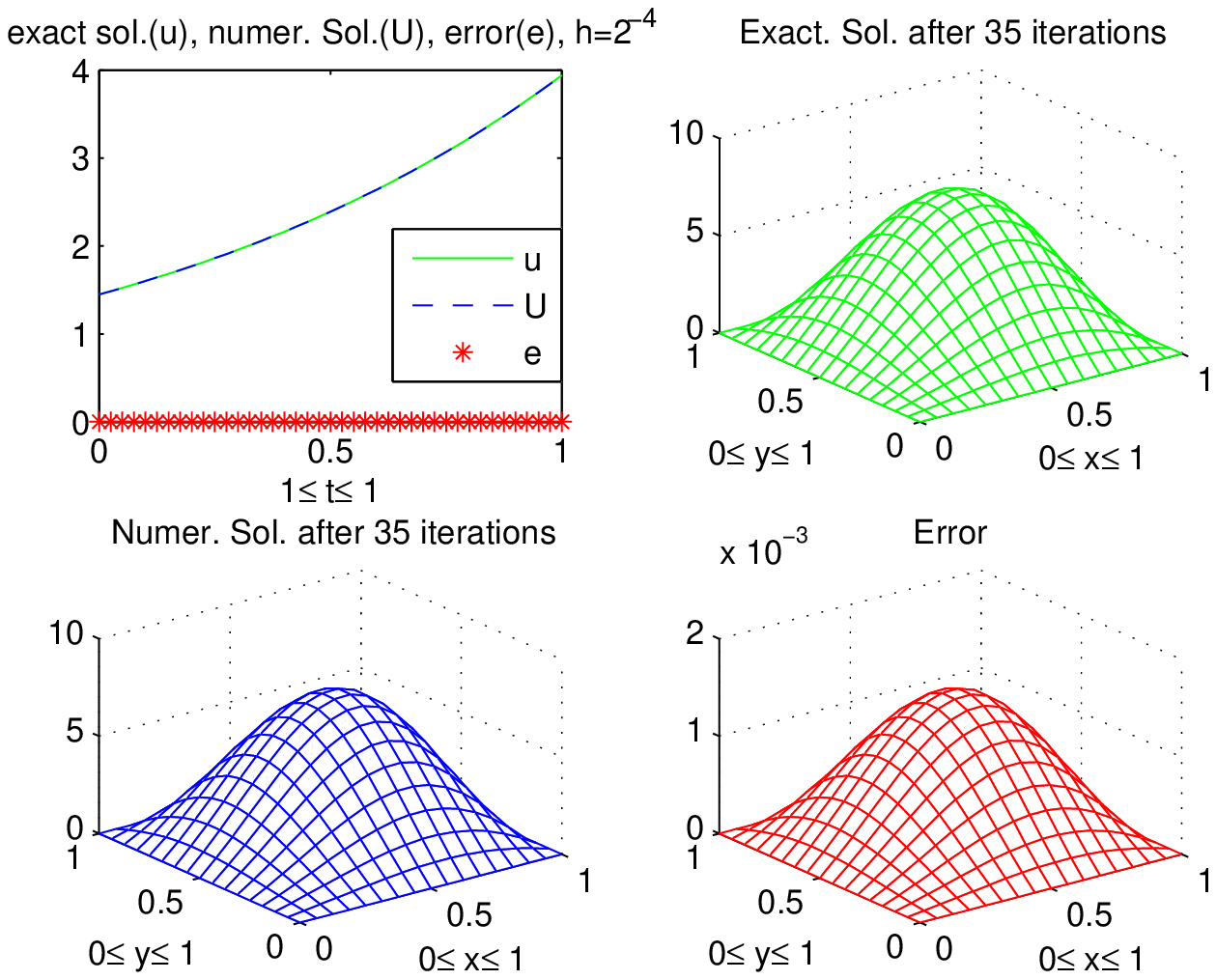,width=7cm}\\
         \end{tabular}
        \end{center}
        \caption{analytical solution(u: in green), computed solution(U: in blue) and error(E: in red) for Example 2}
        \label{figure2}
        \end{figure}

       \begin{figure}
         \begin{center}
          Stability and convergence of the proposed three-level time split explicit/implicit approach with $k=h^{4/3}$
         \begin{tabular}{c c}
         \psfig{file=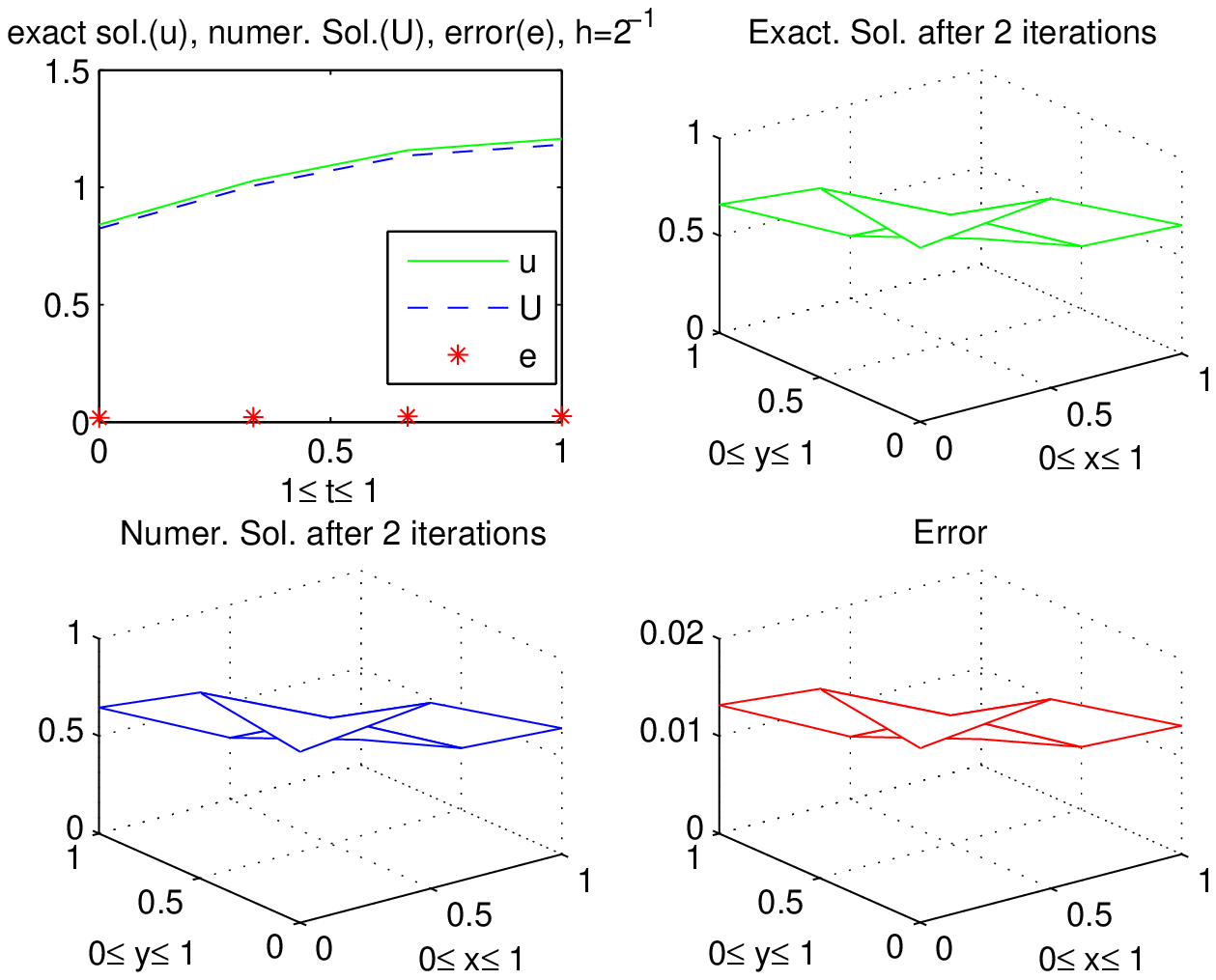,width=7cm} & \psfig{file=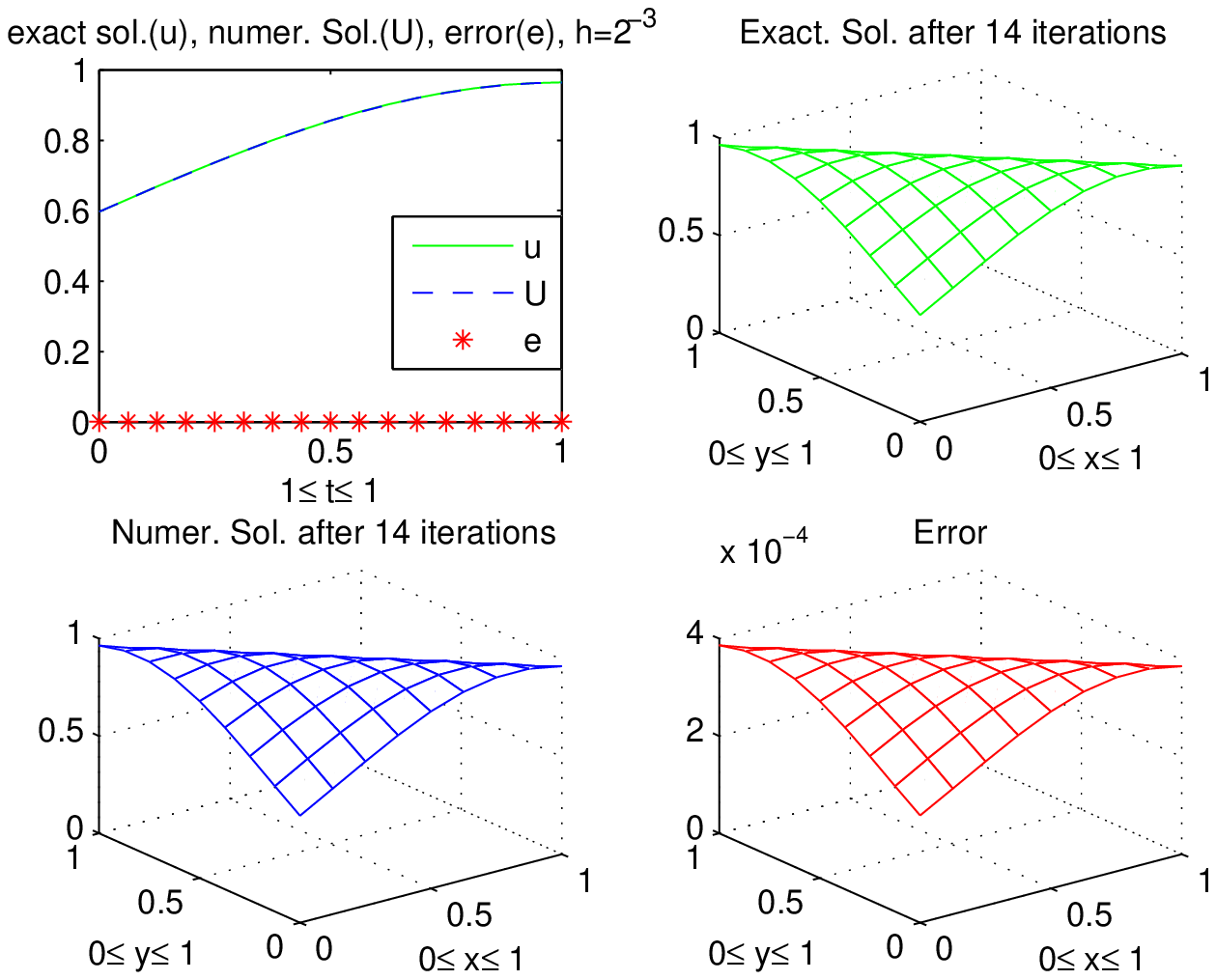,width=7cm}\\
         \psfig{file=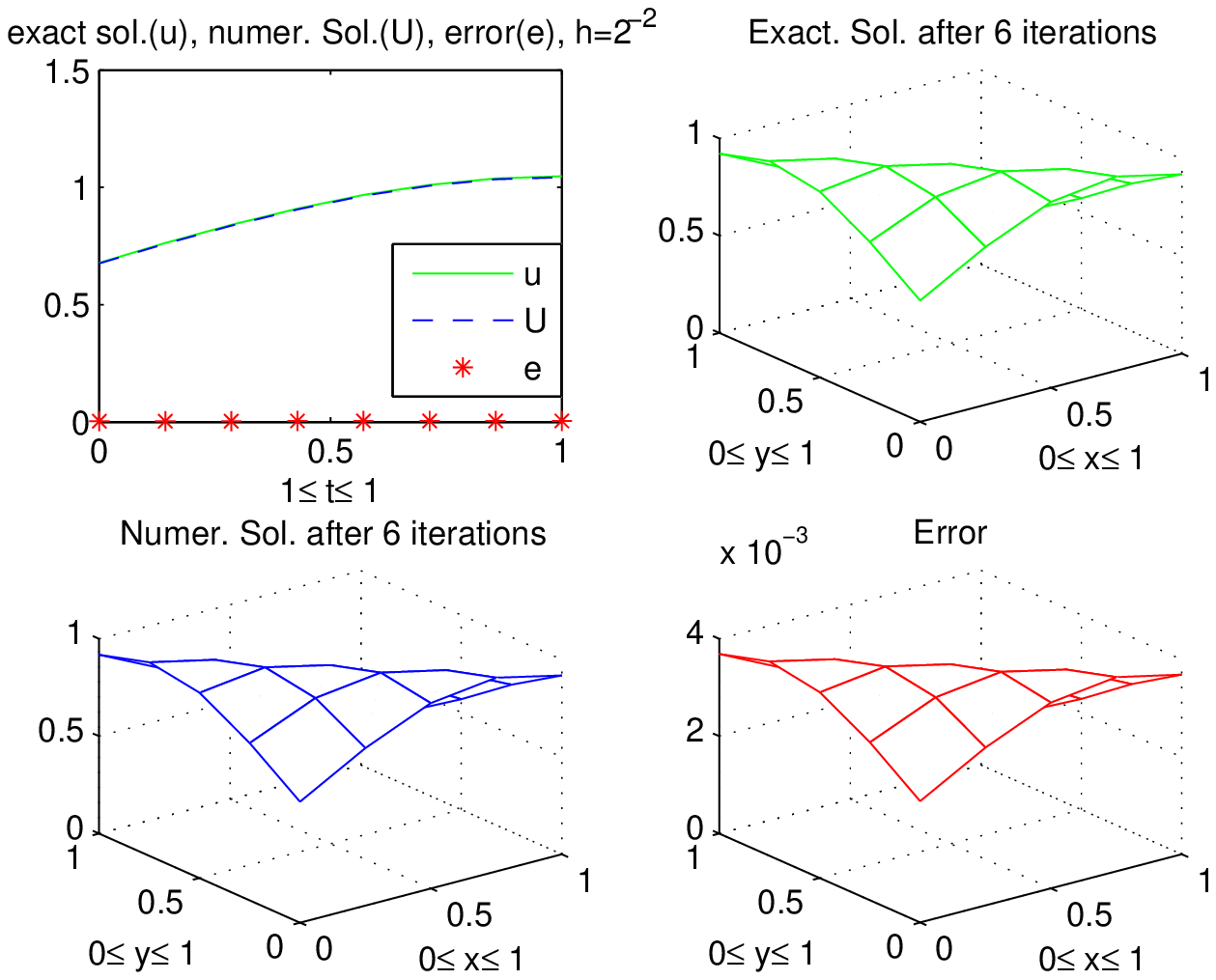,width=7cm} & \psfig{file=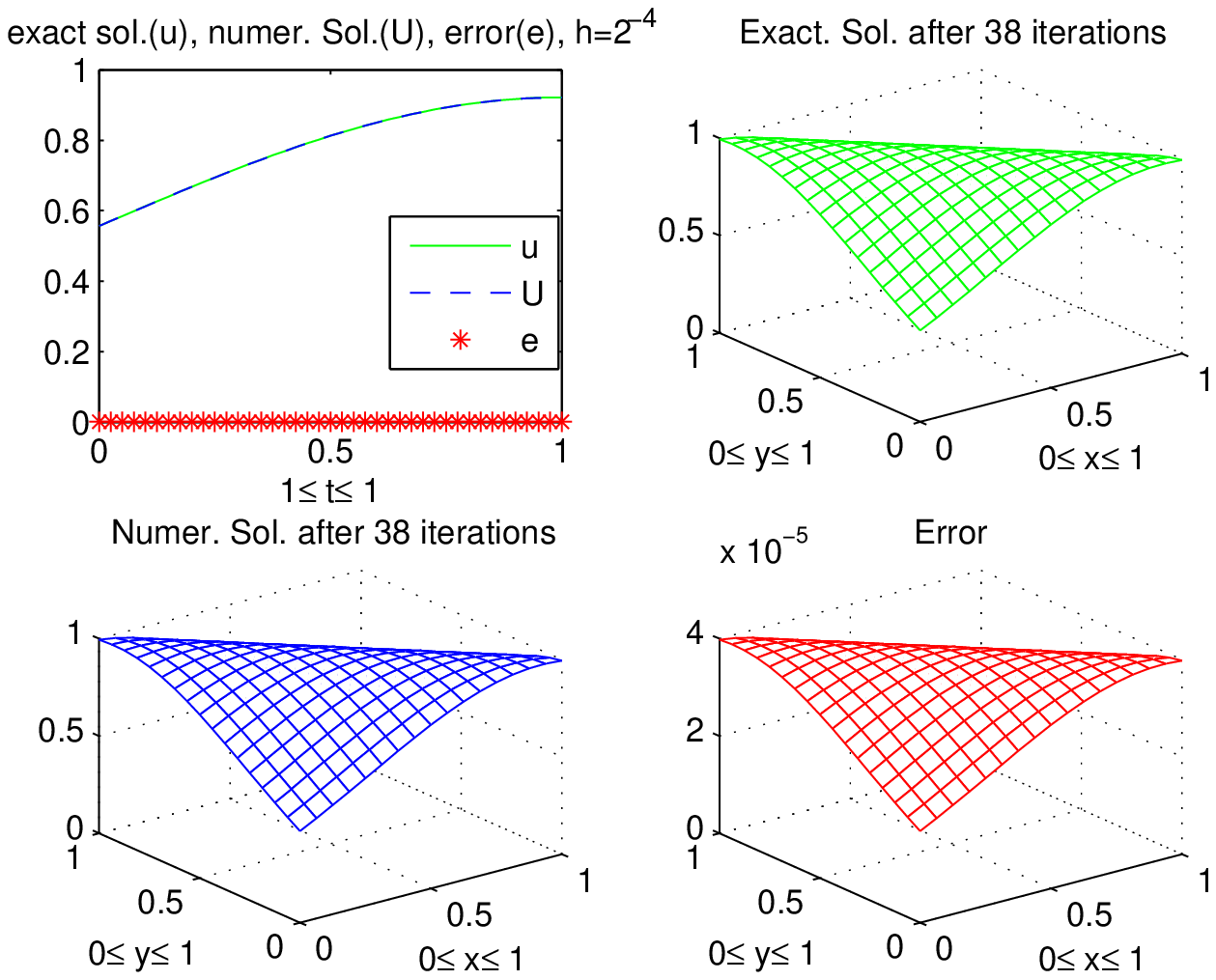,width=7cm}\\
         \end{tabular}
        \end{center}
        \caption{exact solution(u: in green), approximate solution (U: in blue) and error (E: in red) for Problem 3}
        \label{figure3}
        \end{figure}
       \end{document}